\documentclass[11pt]{amsart}
\usepackage{amsmath,amsthm}
\usepackage{graphics}
\usepackage{latexsym}
\usepackage{verbatim}
\usepackage{hyperref}
\usepackage{amssymb}
\usepackage{epsfig}
\usepackage[outdir=./]{epstopdf}
\usepackage{dsfont}
\usepackage[table]{xcolor}
\usepackage{multirow}
\usepackage{float}
\usepackage{tikz}
\usepackage{dcolumn}
\usepackage{bm}
\usepackage{subfigure}
\usepackage{comment}
\usepackage{enumerate}
\usepackage{bbm}
\newcommand{\supp}{\operatorname{supp}}
\renewcommand{\d}[0]{{\mathrm{d}}}

\newcommand{\ddt}{\frac{\d}{\d t}}

\newcommand{\R}[0]{\mathbb{R}}
\newcommand{\N}[0]{\mathbb{N}}
\newcommand{\qqand}[0]{ \qquad{\text{and}} \qquad}
\newcommand{\qand}[0]{ \quad{\text{and}} \quad}

\newcommand{\curlyF}{\mathcal{F}}
\newcommand{\curlyE}{\mathcal{E}}
\newcommand{\curlyP}{\mathcal{P}}
\newcommand{\curlyT}{\mathcal{T}}
\newcommand{\curlyL}{\mathcal{L}}
 
\newcommand{\ds}{\displaystyle}

\newcommand{\partialt}[1]{\dfrac{\partial#1}{\partial t}}

\newcommand{\fpartial}[1]{\dfrac{\partial}{\partial #1}}

\def\obel#1{
\begin{tabular}[t]{c} #1 \end{tabular}}

\newcommand{\eps}{\varepsilon}

\newcommand{\calX}{\mathcal{X}}

\newcommand{\id}{\operatorname{id}}
\newcommand{\matrixid}{\mathbbm{1}}
\newcommand{\indicator}{\mathbbm{1}}

\newcommand{\pt}{\frac{\partial}{\partial t}}
\newcommand{\px}{\frac{\partial}{\partial x}}

\newcommand{\Len}{\curlyF_{\rm{L}}}
\newcommand{\NLen}{\curlyF_{\rm{NL}}}

\newcommand{\curlyO}{\mathcal{O}}
\usetikzlibrary{arrows, calc, shapes, positioning, decorations.pathreplacing}
\usepackage{comment}
\usepackage{subfigure}

\ifpdf
  \DeclareGraphicsExtensions{.eps,.pdf,.png,.jpg}
\else
  \DeclareGraphicsExtensions{.eps}
\fi
\newtheorem{theorem}{Theorem}[section]

\newtheorem{lemma}[theorem]{Lemma}
\newtheorem{proposition}[theorem]{Proposition}

\newtheorem{remark}{Remark}

\title[Segregation effects and Gap Formation in Cross-Diffusion Models]{Segregation and Gap Formation in Cross-Diffusion Models}

\author[Martin Burger, Jos\'e A. Carrillo, Jan-Frederik Pietschmann, Markus Schmidtchen]{}

\subjclass[2010]{35B36, 	35K45, 35K65,	35Q92  }
\keywords{Nonlinear Cross-Diffusion, Degenerate Parabolic Equations, Segregated Solutions, Energy Minimisation, Pattern Formation.}

\email{martin.burger@fau.de}
\email{carrillo@imperial.ac.uk}
\email{jfpietschmann@math.tu-chemnitz.de}
\email{m.schmidtchen15@imperial.ac.uk}

\begin{document}
\maketitle
\centerline{\scshape Martin Burger}
\medskip
{\footnotesize
    \centerline{Department Mathematik, Friedrich-Alexander Universit\"at Erlangen-N\"urnberg}
    \centerline{Cauerstrasse 11, 91058 Erlangen, Germany}
}
\medskip

\centerline{\scshape Jos\'e A. Carrillo}
\medskip
{\footnotesize
    \centerline{Department of Mathematics, Imperial College London}
    \centerline{London SW7 2AZ, United Kingdom}
} 
\medskip

\centerline{\scshape Jan-Frederik Pietschmann}
\medskip
{\footnotesize
    \centerline{Fakult\"at f\"ur Mathematik, Technische Universit\"at Chemnitz}
    \centerline{Reichenhainer Stra\ss{}e 41, Chemnitz, Germany}
}
\medskip

\centerline{\scshape Markus Schmidtchen}
\medskip
{\footnotesize
    \centerline{Department of Mathematics, Imperial College London}
    \centerline{London SW7 2AZ, United Kingdom}
}

\begin{abstract}
In this paper we analyse a class of nonlinear cross-diffusion systems for two species with local repulsive interactions that exhibit a formal gradient flow structure with respect to the Wasserstein metric. We show that systems where the population pressure is given by a function of the total population are critical with respect to cross-diffusion perturbations. This criticality is showcased by proving that adding an extra cross-diffusion term that breaks the symmetry of the population pressure in the system leads to completely different behaviours, namely segregation or mixing, depending on the sign of the perturbation.  We show these results at the level of the minimisers of the associated free energy functionals. We also analyse certain implications of these results for the gradient flow systems of PDEs associated to these functionals and we present a numerical exploration of the time evolution of these phenomena.
\end{abstract}
\section{Introduction}
This paper is dedicated to studying the following system of cross-diffusion equations for two densities $\rho=\rho(x,t), \eta=\eta(x,t)$,
\begin{align}
\label{eq:system}
\begin{split}
	\pt \rho &= \px \left(\rho \px \left(\frac{\delta \curlyF}{\delta \rho}\right)\right) = 
	\fpartial x \left(\rho \fpartial x \left((1+\delta)\rho  + \eta\right)\right) ,\\
	\pt \eta &= \px \left(\eta \px \left(\frac{\delta \curlyF}{\delta \eta}\right)\right) = \fpartial x \left(\rho \fpartial x \left(\rho  + (1+\delta)\eta\right)\right)
\end{split}
\end{align}
on a bounded interval $\Omega = (-L,L)$. Here $\delta\curlyF/\delta \rho$ and $\delta\curlyF/\delta \eta$ denote the formal Fr\'echet derivative of either of the functionals $\curlyF \in \{\Len,\NLen\}$,
\begin{align*}
	\Len(\rho,\eta) = \frac{1+\delta}{2}\int_\Omega	(\rho+\eta)^2\d x - \delta \int_\Omega \rho \eta\d x.
\end{align*}
We also study a non-local variation thereof, given by
\begin{align*}
	\NLen(\rho,\eta) = \frac{1+\delta}{2}\int_\Omega	(\rho+\eta)^2\d x - \delta \int_\Omega \rho (K\star \eta)\d x.
\end{align*}
Here $\delta \in (-1,\infty)$ is a model parameter and we like to think of the kernel $K\in L^1(\R; \R_+)$ as a decaying function, in its radial variable, approximating a Dirac measure with unit mass at the origin. We will refer to the case $\delta=0$ as the symmetric case or the critical case thereafter.

Models of this kind have appeared in many mathematical biology contexts: collective behaviour \cite{TBL06,BCM07,BFH14}, cell adhesion models \cite{PBSG15,MT15,BDFS17,CMSTT19,NeuroBio19}, animal patterning \cite{VS15}, and cancer invasion models \cite{CL05,GC08,DomTruGerCha} to name a few. The main modelling reason of the symmetry $\delta=0$ in the cross-diffusion terms is that the local nonlinear diffusion terms in the system arise from the localised repulsion produced by the total population resistance to be squeezed. In other words, these terms should model volume exclusion or size effects in the underlying particle models, and therefore these effects should be independent of the type of particles under consideration, and consequently, these volume exclusion terms should be symmetric by permutation of the species labels. These models have been widely used together with nonlocal terms in order to show cell sorting by adhesion in mathematical biology \cite{PBSG15,MT15,BDFS17,CMSTT19}. In such systems segregation can be shown rigorously if there are different long-range aggregation forces \cite{BDFS17}. The present work shows that desymmetrising the critical case $\delta=0$ and the gradient flow structure using nonlocal interactions is precisely the source of the richness of patterns obtained in those models in mathematical biology. 

In fact, we show that adding $\delta>0$ cross-diffusion perturbations lead to total mixing of the populations while $-1<\delta<0$ cross-diffusion perturbations lead to segregation in terms of the minimisers, candidates to be stable stationary states of the associated gradient flows, of both the local and the nonlocal perturbations. This is in contrast with the critical case $\delta=0$ in which both mixing and segregation can occur for minimisers, phenomena also observed in our numerical experiments in the corresponding gradient flows. Actually, the local model corresponding to the functional $\curlyF = \Len$ has a diffusion matrix given by
\begin{align*}
	\det 
	\left(
		\begin{array}{cc}
			(1+\delta)\rho & \rho\\
			\eta & (1+\delta)\eta
		\end{array}
	\right)	= \delta (2+\delta) \rho\eta.
\end{align*}
As $\rho,\eta$ are non-negative densities, it is easy to identify three different parameter regimes for $\delta \in (-1,\infty)$ --- the case of $\delta>0$, the critical case of $\delta = 0$ and the case $\delta\in(-1,0)$. The first case is well studied in literature, cf. \cite{Jun15,DFEF17, LM11} since the diffusion matrix is positive definite. The critical case $\delta=0$ has been studied in terms of the free boundaries between segregated initial data in \cite{BGHP85, BGH87,BGH87a,BHIM12}. In \cite{CFSS17} a well-posedness result for the system in one dimension with reaction terms allowing for segregated initial data and global BV-bounds was given by means of splitting and optimal transport techniques, while \cite{Gwiazda2018} recently obtained global existence results in more dimensions under more restrictive assumptions on the reaction terms given by non-increasing functions of the pressure. The criticality of the symmetric case $\delta=0$ is understood in terms of the bifurcation in the overall behaviour both at the level of the energies as well as the PDEs associated to them. Note that $\delta \leq -1$ is not a reasonable case for the parabolic system, since then even the self-diffusion is backward.

In this paper, we mainly focus on the case $\delta \in (-1,0)$, which, to the best of our knowledge, has not been studied in the literature.  In this case the determinant is negative whenever both species mix thus indicating a backward diffusion regime. However, it vanishes if and only if both species are segregated. 
Thus the system is initially well-posed if and only both species are initially segregated. It is the aim of this paper to study this case both analytically and numerically at both the level of the PDE associated to the energies $\curlyF\in\{\Len, \NLen\}$ as well as on the level of the energies themselves. We generalise results in the critical case $\delta=0$ to the local and nonlocal cases for $-1<\delta < 0$, showing that all minimisers of the energies are segregated even showing a positive gap between both species in the nonlocal cases depending on the value of $\delta$. To be precise we summarise our main results in terms of minimisers of the free energies in the next table showing the criticality of the symmetric case $\delta=0$.

\begin{table}[h!]
{\small
\begin{tabular}{c|ccc|ccc}
 \multicolumn{1}{c|}{Free Energy}  & \multicolumn{3}{c|}{\textbf{local}} &  \multicolumn{3}{c}{\textbf{non-local}} \\ \hline
 \multicolumn{1}{c|}{Perturbations}  &  \multicolumn{1}{c|}{$\delta > 0$}   & \multicolumn{1}{c|}{$\delta = 0$} & \multicolumn{1}{c|}{$\delta < 0$}  &  \multicolumn{1}{c|}{$\delta > 0$}  & \multicolumn{1}{c|}{$\delta = 0$} & \multicolumn{1}{c}{$\delta < 0$}\\\hline 
\multicolumn{1}{c|}{} &  \multicolumn{1}{c|}{} & \multicolumn{1}{c|}{} & \multicolumn{1}{c|}{} & \multicolumn{1}{c|}{} & \multicolumn{1}{c|}{} & \\[-1mm]
\textbf{convexity} & \multicolumn{1}{c|}{strictly convex}  & \multicolumn{1}{c|}{convex} & nonconvex & \multicolumn{1}{c|}{strictly convex}  & \multicolumn{1}{c|}{convex} & nonconvex\\[1mm]
 \textbf{minimisers} & \multicolumn{1}{c|}{unique}  & \multicolumn{1}{c|}{at least one } & at least one &\multicolumn{1}{c|}{unique}  & \multicolumn{1}{c|}{at least one} & at least one\\[1mm]
 \textbf{gaps} & \multicolumn{1}{c|}{mixing}  & \multicolumn{1}{c|}{both} &  {segregation} & \multicolumn{1}{c|}{mixing}  & \multicolumn{1}{c|}{both} & \obel{segregation \\ (possibly gaps)}
\end{tabular}
\label{tab:propertiesenergies}
\bigskip}
\caption{Properties of minimisers to the local ($\Len$) and non-local energies ($\NLen$).}
\end{table}

The rest of this paper is organised as follows. In Section \ref{sec:energy}, we study minimisers of the free energies $\mathcal F$ both in the local and nonlocal cases. We analyse the properties of the functionals emphasising the analysis of mixing and segregation phenomena and the study of gaps between the species forming in the non-local case. In the final Section \ref{sec:numerics}, we present extensive numerical results, both for the energy and the associated PDEs discussing several open problems related to the segregation and mixing phenomena in the evolutions.

\section{Properties and minimisers of the free energies}\label{sec:energy}
This section is dedicated to a study of local minimisers of the free energies: the local $\Len$, and its corresponding nonlocal counterpart $\NLen$ where the minimisation problem reads
\begin{align*}
	(\rho, \eta) \in \mathrm{argmin}_{(\bar \rho, \bar \eta) \in \calX}\ \, \curlyF(\bar \rho, \bar\eta),
\end{align*}
with the set of feasible minimisers given by
\begin{align}
\calX = \left\{ (\rho, \eta) \in L_+^2(\Omega)\times L_+^2(\Omega)\;\big|\, \int_\Omega \rho \,\d x= m_1,\, \int_\Omega\, \eta \d x= m_2\right\}.
\end{align} 
The main properties have been summarised in table \ref{tab:propertiesenergies} above. We will analyse these properties in a precise way in the next subsections.

\subsection{The local case -- $\Len$}\label{sec:local_energy}
Let us recall the local energy functional
\begin{align}
	\label{eq:local_energy}
	\Len(\rho,\eta) = \frac{1+\delta}{2}\int_\Omega	(\rho(x)+\eta(x))^2\d x - \delta \int_\Omega \rho(x)\,\eta(x)\d x.
\end{align}
As mentioned in the introduction there are three different parameter regimes $\delta>0$, the critical case $\delta=0$, and finally $\delta\in(-1,0)$. In the case $\delta < -1$ it is easy to see by choosing two blow-up sequences of densities with disjoint support that the infimum of $ \Len$ equals $-\infty$. In the case $\delta = -1$ it is easy to see that infimum is zero and indeed every pair with disjoint support is a minimiser. For the other parameter cases we obtain an existence result:
\begin{theorem}[Existence of minimisers]\label{thm:exminlocal} For $\delta \in (-1,\infty)$, minimisers to the problem \eqref{eq:local_energy} with $\curlyF = \Len$ exist in the set $\mathcal X$ and have the following properties:
\begin{enumerate}
 \item For $\delta > 0$, the unique minimisers of \eqref{eq:local_energy} are given by
\begin{align*}
  \rho = \frac{m_1}{|\Omega|},\qand \eta = \frac{m_2}{|\Omega|}.
\end{align*}
In particular, both species are fully mixed.\\
\item For $\delta = 0$, there exits an infinite family of minimisers to \eqref{eq:local_energy}. Any local minimiser satisfies $\supp \rho \cup \supp \eta = \bar \Omega$ and 
\begin{align*}
 \sigma = \rho + \eta = \frac{m_1 + m_2}{|\Omega|} = \text{const}.
\end{align*}
In particular, both segregation and mixing is possible. Moreover, any local minimiser is a global minimiser.\\
\item For $-1 < \delta < 0$, there exists an infinite family of minimisers to \eqref{eq:local_energy}. Furthermore any minimiser satisfies $\supp \rho \cup \supp \eta = \bar \Omega$ and 
\begin{align*}
\sigma = \rho + \eta = \frac{m_1 + m_2}{|\Omega|} = \text{const}.
\end{align*}
However, there holds $|\supp \rho \cap \supp \eta| = 0$, so that minimisers are always segregated.
\end{enumerate}
\end{theorem}

\begin{proof}
We will address each statement individually.\\

\textbf{Ad (1):}  Let $\delta > 0$. In this case, we note that the energy functional $\Len$ is strictly convex which directly yields existence and uniqueness of the minimiser. As can be seen easily,  $\rho = m_1 |\Omega|^{-1}$ and $\eta = m_2 |\Omega|^{-1}$ are critical points of the energy which yields the first statement.\\[1em]

\textbf{Ad (2):} Let $\delta = 0$. We begin with the properties of the minimisers and show existence later. To this end we  assume that $(\rho, \eta)\in \calX$ is a minimiser of the energy $\Len$ such that  there exists an open set
	\begin{align*}
		\mathring{D} \subset \bar \Omega\setminus(\supp(\rho) \cup \supp(\eta)),
	\end{align*}
	i.e., there are regions ssof vacuum in $\bar \Omega$. Furthermore let $B_1\subset \supp(\rho)$, $B_2 \subset \supp(\eta)$, and $B_3 \subset \mathring{D}$ be sets of equal Lebesgue measure such that
	\begin{align*}
		\rho\big|_{B_1} > \alpha, \quad \eta\big|_{B_2} > \alpha, \qand \rho\big|_{B_3} = \eta\big|_{B_3} =0,
	\end{align*}
	almost everywhere for some $\alpha>0$. Furthermore we set
	\begin{align*}
		\rho^\epsilon &= \rho - \epsilon \matrixid_{B_1} + \epsilon \matrixid_{B_3},\\
		\eta^\epsilon &= \eta - \epsilon \matrixid_{B_2} + \epsilon \matrixid_{B_3},
	\end{align*}
	for $0<\epsilon<\alpha$. Note that the perturbed minimisers have the same mass and remain non-negative. In addition, we note that $\rho^\epsilon + \eta^\epsilon = \rho + \eta$.

	Then there holds
	\begin{align*}
		\Len(\rho^\epsilon, \eta^\epsilon) &= \frac12 \int_{\Omega\setminus(B_1\cup B_2\cup B_3)} \sigma^2 \d x +\frac12 \int_{B_1} (\rho -\epsilon + \eta)^2\d x +\frac12 \int_{B_2} (\rho+\eta-\epsilon)^2\d x +\frac12 \int_{B_3} (\epsilon +\epsilon)^2\d x\\
		&=\Len(\rho, \eta) - \epsilon \left(\int_{B_1}\sigma \d x+ \int_{B_2}\sigma\d x \right) + 3 \epsilon^2\,|B_3|\\
		&< \Len(\rho, \eta), 
	\end{align*}
	for $\epsilon>0$ small enough. 
	Thus we have constructed a better minimiser which is a contradiction. Hence any minimiser, $(\rho, \eta)$, occupies the entire domain up to a set of measure zero, i.e., $\supp(\rho) \cup \supp(\eta) = \Omega$ up to a set of zero Lebesgue measure. Finally, let us show that any minimiser satisfies
\begin{align}
   	\label{eq:necessary_condition_minr}
	\sigma = \frac{m_1 + m_2}{|\Omega|}.
\end{align}
In order to obtain more information, we need to deduce Euler-Lagrange conditions for local minimisers of the free energies. We will do suitable perturbations of the densities following the blueprints of \cite{Strohmer2008,BCLR,CCV,CHVY,CCH16a,CCH16b,CHMV} in related cases. Let us begin by computing variations of the energy with respect to the first species, $\rho$.
\begin{align}
	\label{eq:ELequation}
	\frac1\epsilon \left(\Len(\rho + \epsilon \phi, \eta) - \Len(\rho, \eta)\right) = 2\int (\rho + \eta) \phi \d x + \curlyO(\epsilon).
\end{align}
The perturbation must be chosen carefully lest the positivity and the mass constraint be violated. To this end let us set
\begin{align*}
	\phi(x) = \rho(x) \left(\Psi(x) - \frac{1}{m_1} \int_\Omega \Psi(y)\rho(y)\d y\right),
\end{align*}
for some $\Psi \in C_0^\infty(\Omega)$ and $0<\epsilon< \|\Psi\|_{L^\infty}^{-1}$. Substituting this into the first variation we obtain
\begin{align*}
	0 &=2\int_\Omega (\rho + \eta) \rho(x) \Psi(x) - \frac{1}{m_1}\rho(x) (\rho(x) + \eta(x))\int_\Omega \Psi(y)\rho(y)\d y \d x \\
	&= 2 \int_\Omega \Psi(x) \rho(x)\left[ (\rho(x) + \eta(x)) - \frac{1}{m_1}\int_\Omega (\rho(y) + \eta(y)) \rho(y) \d y\right] \d x,
\end{align*}
whence, on $\supp(\rho)$, we have
\begin{align*}
	\rho + \eta  = \frac{1}{m_1} \int_\Omega (\rho + \eta) \rho \d y. 
\end{align*}
A similar perturbation with respect to the second species yields
\begin{align*}
	\rho + \eta  = \frac{1}{m_2} \int_\Omega (\rho + \eta) \rho \d y,
\end{align*}
on $\supp(\eta)$.

Next, we choose a different perturbation, $\phi$, in \eqref{eq:ELequation} in order to obtain information of the minimisers outside of their respective supports. To this end we set
\begin{align*}
	\phi(x) = m_1 \Psi(x) - \rho(x) \int_\Omega \Psi(y)\d y,
\end{align*}
for some $\Psi\in C_0^\infty(\Omega)$, which is preserves the mass and positivity of $\rho$ if $0<\epsilon<\|\Psi\|_{L^1}^{-1}$. Using this in the Euler-Lagrange equation yields
\begin{align*}
	0 &\leq 2\int_\Omega (\rho(x)+\eta(x))m_1 \Psi(x)- (\rho(x)+\eta(x))\rho(x) \int_\Omega \Psi(y)\d y\d x \\
	&= 2 \int_\Omega \Psi(y)\left[\sigma(y) m_1 - \int_\Omega \sigma(x) \rho(x)\d x \right]\d y,
\end{align*}
and we conclude
\begin{align*}
	\frac{1}{m_1} \int_\Omega \rho^2(x)\d x \leq 	\frac{1}{m_1} \int_\Omega \sigma(x)\rho(x) \d x \leq \sigma(x) = \eta(x),
\end{align*}
for almost every $x \notin \supp(\rho)$. Similarly, we obtain
\begin{align*}
	\frac{1}{m_2} \int_\Omega \eta^2(x)\d x \leq 	\frac{1}{m_2} \int_\Omega \sigma(x)\eta(x) \d x \leq \sigma(x) = \rho(x),
\end{align*}
for almost ever $x \notin \supp(\eta)$. In summary, we have
\begin{align}
	\label{eq:2405191535}
	\left\{
	\begin{array}{rl}
		(\rho + \eta) = \ds \frac{1}{m_1} \int_\Omega \sigma \rho, & \mbox{in } \supp(\rho)\\[1em]
		(\rho + \eta) \geq \ds \frac{1}{m_1} \int_\Omega \sigma \rho, & \mbox{outside } \supp(\rho)\\[1.9em]
		(\rho + \eta) = \ds \frac{1}{m_2} \int_\Omega \sigma \eta, & \mbox{in } \supp(\eta)\\[1em]
		(\rho + \eta) \geq \ds  \frac{1}{m_2} \int_\Omega \sigma \eta, & \mbox{outside } \supp(\eta)
	\end{array}
	\right.
\end{align}
Finally, we perturb the functional in both variables at the same time, i.e.,
\begin{align}
	\frac1\epsilon \left(\Len(\rho + \epsilon \phi, \eta + \epsilon \phi) - \Len(\rho, \eta)\right) = 2\int (\rho + \eta) \phi \d x + \curlyO(\epsilon).
\end{align}
As before, choosing
\begin{align*}
	\phi(x) = (\rho +\eta) \left(\Psi(x) - \frac{1}{m_1 + m_2}\int_\Omega \Psi(y) (\rho(y)+\eta(y))\right),
\end{align*}
for some $\Psi\in C_0^\infty(\Omega)$ and $0< \epsilon < \|\Psi\|_{L^\infty}^{-1}$. Unlike the individual perturbation above, perturbing in the support of $(\rho + \eta)$ yields, by a computation similar to the one before, 
\begin{align*}
	\rho(x) + \eta(x) = \frac{1}{m_1 + m_2} \int_\Omega (\rho(y) + \eta(y))^2 \d y.
\end{align*}
Thus, there holds $\sigma(x) = (m_1 + m_2) / |\supp(\sigma)|$ on the support of $\sigma$.
Note that this implies that the constants in the individual Euler-Lagrange equations corresponding to $\rho$ and $\eta$ (see \eqref{eq:2405191535}) are indeed the same since
\begin{align*}
	\frac{1}{m_1}\int_\Omega \sigma \rho \d x = \frac{1}{m_1}\frac{m_1 + m_2}{|\supp(\sigma)|} \int_\Omega \rho \d x = \frac{1}{m_2}\frac{m_1 + m_2}{|\supp(\sigma)|} \int_\Omega \eta \d x = \frac{1}{m_2} \int_\Omega \sigma \eta \d x.
\end{align*}
Using the fact that $\supp(\sigma) = \Omega$ we have shown that any local minimiser of $\Len$ satisfies $\sigma=(m_1 + m_2)/|\Omega|$. 
Finally, let us show that any local minimiser of $\Len$ is a global minimiser. To this end consider a local minimiser $(\rho, \eta)$ and perturb it by $f,g\in L^2(\Omega)$ satisfying $f\neq -g$, as well as
\begin{align*}
	\int_\Omega f\,\d x = \int_\Omega g \,\d x= 0,	
\end{align*}
and 
\begin{align*}
  	\tilde \rho = \rho + f \geq 0\qand \tilde \eta = \eta + g \geq 0.
\end{align*}
But then
\begin{align*}
    \Len(\tilde \rho, \tilde \eta) &= \Len(\rho, \eta) + 2\int (\rho+\eta)(f+g)\d x + \int (f+g)^2\d x \\
    &=\Len(\rho, \eta) + 2 \frac{m_1 + m_2}{|\Omega|} \int  (f+g) \d x + \int (f+g)^2\d x \\
    &=\Len(\rho, \eta) + \int (f+g)^2\d x > \Len(\rho, \eta).
\end{align*}
Thus any feasible perturbation leads to a strict increase in the energy. \smallskip\\

\textbf{Ad (3):}  Let $\delta<0$. We start by showing existence of minimisers first. Let $(\rho, \eta)$ be any segregated minimiser of the energy
        \begin{align}
        		\label{eq:2405191527}
            \curlyE(\rho, \eta):=\frac{1+\delta}{2}\int_\Omega (\rho + \eta)^2 \d x
        \end{align}
        We claim that, as a matter of fact, $(\rho, \eta)$ is a minimiser of the functional $\Len$. Assume the contrary and let $(\tilde \rho, \tilde \eta)$ be a competitor that is not a minimiser of \eqref{eq:2405191527}, i.e.,
        \begin{align*}
                \Len(\tilde \rho, \tilde \eta) < \Len(\rho, \eta).
        \end{align*}
        Upon rearranging the terms there holds
        \begin{align*}
                0 \leq -\delta\int_\Omega \tilde \rho \tilde\eta\d x < -\delta \int \rho \eta \d x + \curlyE(\rho, \eta) - \curlyE(\tilde \rho, \tilde \eta) < 0,
        \end{align*}
        which is absurd. Thus $(\rho, \eta)$ is indeed a minimiser of $\Len$. Let us now show the properties that minimisers are segregated. To this end, let $(\rho, \eta)\in \calX$ be a minimiser of $\Len$. Assume $|\supp(\rho)\cap \supp(\eta)|>0$ with respect to the Lebesgue measure. 
        
\noindent
        Let us choose two disjoint sets of equal Lebesgue measure, $B_1, B_2$, in this region, i.e.,
	\begin{align}
		B_1 \dot \cup B_2 \subset \supp(\rho)\cap \supp(\eta),
	\end{align}
	such that $\rho, \eta > \epsilon>0$, for some $\epsilon>0$. As before, we set
\begin{align*}
	\rho^\epsilon &= \rho - \epsilon\matrixid_{B_1} +\epsilon\matrixid_{B_2},\\
	\eta^\epsilon &= \eta + \epsilon\matrixid_{B_1} -\epsilon\matrixid_{B_2},
\end{align*}
observing that the mass, non-negativity, and the sum remain conserved. The energy of the perturbation then reads
\begin{align*}
	\Len(\rho^\epsilon, \eta^\epsilon) &= \frac12 \int_\Omega(\rho +\eta)^2\d x - \delta \int_{\Omega\setminus(B_1 \cup B_2)} \rho\eta\d x\\
	&\quad - \delta \int_{B_1} (\rho - \epsilon)(\eta + \epsilon)\d x - \delta \int_{B_2} (\rho+\epsilon)(\eta - \epsilon)\d x \\
	&=\Len(\rho,\eta) -\delta \epsilon\left(\int_{B_1} (\eta-\rho) \d x + \int_{B_2}(\rho-\eta) \d x\right) + \delta \epsilon^2 (|B_1| + |B_2|).
\end{align*}
Thus, either the order $\epsilon$-terms vanish in which case we get a contradiction straight away as the order $\epsilon^2$-term is negative. Otherwise, the order $\epsilon$-term has a sign (and by possible switching $\epsilon\leftarrow -\epsilon$ we get a contradiction again, for we have constructed a better minimiser, which is absurd. 

Thus, energy minimisers are strictly segregated. By a similar argument, we can see that the minimisers are supported on the whole domain, for otherwise the energy could be decreased by shifting mass to void regions.
\end{proof}
This result is interesting in that it shows that the parameter choice $\delta=0$ is critical for the behaviour. Phase segregation is actually energetically favoured whenever $\delta \in (-1, 0)$ unlike in the case $\delta=0$ (cf.  \cite{BGH87,BGH87a,BHIM12, CFSS17}), where areas of coexistence may be observed. In the case of $\delta > 0$ only states that are completely mixed are preferred by the energy. In this sense $\delta = 0$ is borderline for the qualitative properties of minimiser which we illustrate in Figure \ref{fig:PDElocal_deltas}  for $\Len$ and Figure \ref{fig:PDEnonlocal} for $\NLen$, respectively..

\begin{proposition}[\label{prop:nonlsc}The energy $\Len$ is not weakly lower semicontinuous]
Let $-1<\delta<0$. Then the energy is not weakly lower semicontinuous.
\end{proposition}
\begin{proof}
Without loss of generality we assume $\Omega= [0,1]$. The argument we employ can easily be generalised to any dimension and domain by following the same procedure. We construct two sequences
\begin{align*}
	\rho^n = \sum_{i=0}^{n-1} \matrixid_{[2i/(n+1),2i+1/(n+1))} \qand \eta^n = 1- \rho^n.
\end{align*}
By construction it is apparent that $\supp(\rho^n)\cap \supp(\eta^n)=\emptyset$ for all $n\in\N$. However, let us note that $\rho^n+\eta^n \equiv 1$ for all $n\in\N$. Hence we have constructed a sequence of minimisers converging weakly to the constant function
\begin{align}
	\rho^n,\eta^n\rightharpoonup 1/2, \qquad \text{in}\quad L^2(0,1).
\end{align}
However, the limiting function are fully mixed and in particular no minimisers by Theorem \ref{thm:exminlocal}. Hence the energy $\Len(\rho,\eta)$ is not sequentially weakly lower semicontinuous.
\end{proof}

\subsection{The non-local case -- $\NLen$}
This section is devoted to the study of the energy
\begin{align}
	\label{eq:nonlocenergy}
	\NLen(\rho,\eta)=  \frac{1+\delta}{2}\int (\rho+\eta)^2\d x  -  \delta \iint K(x-y)\rho(x)\eta(y)\d x \d y.
\end{align}
In the non-local case we have the following existence result, valid for all $\delta$-regimes: 
    \begin{lemma}[Existence of minimisers]\label{lem:nonlocal_deltaneg_existence_minimisers}
    	Let $\delta \in (-1,1]$ and $K\in L^1(\Omega)$ with $\|K\|_{L^1(\Omega)}=1$. Then the functional $\NLen$ has at least one minimiser.
    \end{lemma}
    \begin{proof}
    Our proof is based on the direct method of calculus of variations and we treat the cases $\delta \in (-1,0]$ and $\delta \in (0,\infty)$ seperately.\smallskip\\
    \textbf{Case $\delta \in (-1,0]$:} In this case, $\mathcal F_{NL}$ is bounded from below by zero. Thus we can consider a minimising sequence, $(\rho_n, \eta_n)$ which we can choose to be segregated without restriction as they can otherwise be no minimisers by Lemma \ref{lem:segregated_minimisers_delta_neg}. There holds
    	\begin{align*}
    		\NLen(\rho_n, \eta_n) \leq \NLen(\rho_0, \eta_0),
    	\end{align*}
    	for any $n\in \N$. Hence, upon rearranging some terms, we obtain
		\begin{align}
		\label{eq:l2estpre}
			\frac{1+\delta}{2}\int_\Omega \rho_n^2 + \eta_n^2 \, \d x \leq \NLen(\rho_0, \eta_0) + \delta \int_\Omega \rho_n K\star\eta_n\, \d x.
		\end{align}
		Using Young's inequality for convolutions and the weighted Young inequality for products, we have 
		\begin{align}\label{eq:estimate_nonlocal_gamma}
		\begin{split}
		 \delta \int_\Omega \rho_n K\star\eta_n\, \d x &\le |\delta|\|\rho_n\|_{L^2(\Omega)} \|K\star \eta_n\|_{L^2(\Omega)} \le |\delta|\|\rho_n\|_{L^2(\Omega)} \|K\|_{L^1(\Omega)}\|\eta_n\|_{L^2(\Omega)} \\
		 &\le \frac{\delta^2}{2\gamma}\|\rho_n\|_{L^1(\Omega)}^2 + \frac{\gamma}{2} \|\eta_n\|_{L^2(\Omega)}^2.
		 \end{split}
		\end{align}
		As $\delta > -1$ is fixed, there exists $\alpha > 0$ such that $(1+\delta)/2 \ge \alpha$. Choosing $\gamma < \alpha$, the second term can be absorbed by the left hand side in \eqref{eq:l2estpre} while the first term is bounded which yields a uniform $L^2$ estimate on the minimising sequence.\\
    \textbf{Case $\delta \in (0,\infty)$:} To obtain a lower bound we observe that, for all $(\rho,\eta) \in \mathcal X$ there holds
		\begin{align*}
		 	\frac{1+\delta}{2}&\int_\Omega (\rho+\eta)^2\d x - \delta \int_\Omega \rho K \star \eta\d x\\
			 &\ge \frac{1+\delta}{2}\int_\Omega (\rho+\eta)^2\d x-\frac{\delta}{2}\left(\|\rho\|^2_{L^2(\Omega)} + \|\eta\|^2_{L^2(\Omega)}\right)\\
			 &=\frac{1}{2}\int_\Omega (\rho+\eta)^2\d x + \frac{\delta}{2}\int_\Omega \rho\eta \d x \geq 0,
		\end{align*}
        where we used that $K$ is normalised in $L^1$ and that both $\rho$ and $\eta$ are non-negative. In order to obtain a uniform $L^2$ estimate, we observe that
        \begin{align*}
         &\frac{1+\delta}{2}\int_\Omega \rho_n^2+\eta_n^2\d x\le \frac{1+\delta}{2}\int_\Omega (\rho_n+\eta_n)^2\d x \le \NLen(\rho_0,\eta_0) +  \delta \int_\Omega \rho_n K \star \eta_n\d x\\
         &\leq \NLen(\rho_0,\eta_0) +  \delta \|\rho_n K \star \eta_n\|_{L^1(\Omega)}.
        \end{align*}
        Estimating the last term as in \eqref{eq:estimate_nonlocal_gamma} and arguing as above, we obtain the desired bounds on $(\rho_n,\eta_n)$.
		Thus we obtain, in both cases, a lower bound on the functional and uniform $L^2$--bounds on the minimising sequence which yields the existence of two functions $\rho, \eta \in L^2(\Omega)$ such that the sequence converges weakly to the pair $(\rho, \eta)$. The non-negativity of the limits is a consequence of the weak convergence, as is the conservation of mass, as the sequence converges weakly in $L^1(\Omega)$ as well. By the lower-semicontinuity of the functional $\NLen$, we obtain
		\begin{align*}
			\NLen(\rho, \eta) \leq \liminf_{n\rightarrow \infty}\NLen(\rho_n, \eta_n),
		\end{align*}
		which concludes the proof. Notice that the nonlocal part of the functional is even continuous with the hypothesys on the kernel $K$.
    \end{proof}
For $\delta < 0$, we have the following segregation property.
\begin{lemma}[\label{lem:segregated_minimisers_delta_neg}Minimisers are strictly segregated] Let $-1< \delta < 0$, and $(\rho,\eta)$ a minimiser of \eqref{eq:nonlocenergy}. Then both densities are segregated up to a set of measure zero, i.e., there holds
	\begin{align*}
		\supp(\rho) \cap \supp(\eta) = \emptyset,
	\end{align*}
    up to a set of measure zero.
    \end{lemma}
    \begin{proof}
		The proof is by contradiction, so let us assume that both densities are supported on a set of positive Lebesgue measure, $\mathring D \subset \supp(\rho)\cap \supp(\eta)$. Without restriction (possibly by restricting $\mathring D$ further) we may assume that  $\rho\big|_{\mathring D},\, \eta\big|_{\mathring D}>\alpha$, for some $\alpha>0$.
		
		Next we choose two non-intersecting sets $B_1, B_2 \subset \mathring D$ with equal Lebesgue measure, \emph{i.e.}, $|B_1|=|B_2|$, and define the perturbed minimisers by
		\begin{align*}
			\rho^\epsilon &:= \rho - \epsilon \matrixid_{B_1} + \epsilon\matrixid_{B_2},\\
			\eta^\epsilon &:= \eta + \epsilon \matrixid_{B_1} - \epsilon\matrixid_{B_2}.
		\end{align*}
		Note that the mass is invariant under this perturbation and non-negativity of the competitors is guaranteed as long as $\epsilon < \alpha$. For ease of notation we set $B^\star := B_1 \times B_1$, and $B_\star = B_2 \times B_2$, and also $B:= B^\star \cup B_\star$.
		Then the perturbed energy reads
		\begin{align*}
			\NLen(\rho^\epsilon, \eta^\epsilon) &= \frac{1+\delta}{2} \int_\Omega \left(\sigma^\epsilon \right)^2 \d x- \delta \int_\Omega \eta^\epsilon K\star \rho^\epsilon\d x\\
			&= \frac{1+\delta}{2}\int_{\Omega}\sigma^2 \d x - \delta \iint_{\Omega^2\setminus B}K(x-y)\rho(y) \eta(x) \d x \d y - \delta \iint_{B}K(x-y)\rho^\epsilon(y) \eta^\epsilon(x)\d x \d y,
		\end{align*}
		having used the fact that $\rho^\epsilon + \eta^\epsilon  = \rho + \eta$.
		Finally, we address the cross-interaction term.
		\begin{align*}
			\iint_{B}  &K(x-y) \rho^\epsilon(y)\eta^\epsilon(x) \d x \d y \\
			&= \iint_{B^\star}  K(x-y) \rho^\epsilon(y)\eta^\epsilon(x) \d x \d y + \iint_{B_\star} K(x-y) \rho^\epsilon(y)\eta^\epsilon(x)\d x \d y \\
			&=\iint_{B^\star} K(x-y) (\rho(y)-\epsilon)(\eta(x)+\epsilon)\d x \d y + \iint_{B_\star}  K(x-y) (\rho(y)+\epsilon)(\eta(x) - \epsilon)\d x \d y.
		\end{align*}
		The first integral becomes
		\begin{align*}
			\iint_{B^\star} &K(x-y)(\rho(y)-\epsilon)(\eta(x) + \epsilon)\d x \d y \\ 
			&=  \iint_{B^\star} K(x-y)\rho(y)\eta(x) \d x\d y+ \epsilon \iint_{B^\star} K(x-y) (\rho(y)-\eta(x))\d x \d y  + \epsilon^2\iint_{B^\star}K(x-y)\d x \d y,
		\end{align*}
		and, similarly, the second term becomes
		\begin{align*}
			\iint_{B_\star}  &K(x-y) (\rho(y)+\epsilon)(\eta(x) - \epsilon)\d x \d y\\
			&= \iint_{B_\star} K(x-y)\rho(y)\eta(x)\d x \d y + \epsilon\iint_{B_\star}K(x-y) (\eta(x)-\rho(y))\d x \d y + \epsilon^2\iint_{B_\star} K(x-y) \d x \d y.
		\end{align*}
		Upon combining all these computations we get
		\begin{align*}
			\NLen(\rho^\epsilon, \eta^\epsilon) &= \NLen(\rho, \eta)\\
			&\quad + \epsilon\delta \bigg\{\iint_{B^\star} K(x-y) (\rho(y)-\eta(x)) \d x \d y + \iint_{B_\star} K(x-y) (\eta(x)-\rho(y))\d x\d y\bigg\}\\
			&\quad+ \epsilon^2 \delta \iint_B K(x-y)\d x \d y.
		\end{align*}
		If $\rho=\eta$ a.e. on $B_1\cup B_2$ then we already get a contradiction, for the order $\epsilon$-terms vanish and the last term is clearly negative.
		Thus, let us assume there exists some  $\alpha'>0$ and a subset $P\,\dot\cup\, N \subset B_1 \cup B_2$. such that
			$\rho - \eta > \alpha'$, 
		almost everywhere on $P$, and, 
			$\rho - \eta < \alpha'$,
		almost everywhere on $N$, respectively.
		 We choose two more sets of equal Lebesgue measure, again labelled $B_1\subset P$ and $B_2 \subset N$. Upon performing the same computation as above we obtain
		 \begin{align}
		 	\NLen(\rho^\epsilon, \eta^\epsilon) \leq  \NLen(\rho, \eta) + \epsilon \delta \alpha' \iint_B K(x-y) \d x \d y  + \epsilon^2 \delta \iint_B K(x-y) \d x \d y < \NLen(\rho,\eta),
		 \end{align}
		 which is a contradiction as $(\rho, \eta)$ was assumed to be a minimiser.
    \end{proof}
    Next we show that in the case of positive $\delta$ energy minimisers are fully mixed.
        \begin{lemma}[Minimisers are fully mixed for $\delta >0$] Let $\delta > 0$, $K\in L_+^1(\R)$,  and $(\rho,\eta)$ a minimiser of the energy \eqref{eq:nonlocenergy}. Then both densities are fully mixed, i.e.
     \begin{align}
     	\supp(\rho) \cap \supp(\eta) = \Omega,
     \end{align}
     up to a set of measure zero.
    \end{lemma}
    \begin{proof}
    	Again assume that a minimiser of the energy is not fully mixed, i.e. assume that there exists an open set $\mathring D \subset \supp(\rho)$ s.t. $\eta\big|_{\mathring D} \equiv 0$.
        Now let $B_1 \subset \mathring D$ and $B_2$ such that $\eta-\rho > 0$ on it. A computation similar to that  above yields
        \begin{align*}
        \NLen(\rho^\epsilon, \eta^\epsilon) &= \NLen(\rho, \eta)\\
        &\quad + \epsilon \delta \bigg\{\iint_{B^\star} K(x-y)\rho(y) \d x \d y+ \iint_{B_\star} K(x-y)(\eta(x)-\rho(y))\d x \d y\bigg\} + \curlyO(\epsilon^2).
        \end{align*}
        The term of order $\epsilon$ is positive, whence
        \begin{align*}
        	\NLen(\rho^\epsilon, \eta^\epsilon) < \NLen(\rho, \eta),
        \end{align*}
        for small $-\epsilon > 0$. This is a contradiction and concludes the proof.
    \end{proof}

\section{Analysis of the Cross-Diffusion System}\label{sec:pde}

In this section we shall list and discuss properties of the evolution equations 
\begin{align}
\label{eq:PDEsystem_local}
\left\{
\begin{array}{rll}
	\dfrac{\partial \rho}{\partial t} &= \fpartial x \left(\rho \fpartial x \left((1+\delta)\rho  + \eta\right)\right),& \text{in } \Omega,\\[1.2em]
	\dfrac{\partial \eta}{\partial t} &= \fpartial x \left(\eta \fpartial x \left(\rho  + (1+\delta)\eta\right)\right), & \text{in } \Omega,\\[1.2em]
	&\rho \fpartial x \left((1+\delta)\rho + \eta\right) =0,& \text{at } x=\pm L, \\[0.6em]
	&\eta \fpartial x \left(\rho + (1+\delta) \eta\right) =0 & \text{at } x = \pm L,
\end{array}
\right.
\end{align}
depending on $\delta$. System \eqref{eq:PDEsystem_local}  is equipped with initial data $\rho(0,x) = \rho_0(x)$ and $\eta(0,x) = \eta_0(x)$ for two non-negative, square-integrable functions $\rho_0, \eta_0 \in L_+^2(-L,L)$. As already set out in the introduction the behaviour of system \eqref{eq:PDEsystem_local} is largely dependent on the choice of $\delta$. 
\subsection{The case $\delta > 0$}
In the case of $\delta > 0$ it is easily verified that the diffusion matrix is coercive and the existence theory of \cite{Jun15} or \cite{DFEF17}, for instance, can be applied. In the latter case, we note that the energy density satisfies the coercivity assumptions (D1)-(D3) with $m_1 = m_2 = 2$ so that $\alpha_1,\alpha_2,a,b$ can be chosen suitably. Then $\rho,\eta \in L^6(0,T; \R)$ and $\nabla \rho,\nabla \eta \in L^2((0,T)\times \R)$ such that
\begin{align*}
	\left\{
	\begin{array}{rl}
	\ds \ddt \int_\R \phi(x) \rho(x)\d x &\ds = -\int_\R \rho \nabla \frac{\delta \Len}{\delta \rho}  \cdot \nabla \phi \d x,\\[1.2em]
	\ds\ddt \int_\R \phi(x) \rho(x)\d x &\ds = -\int_\R \eta \nabla \frac{\delta \Len}{\delta \eta}  \cdot \nabla \phi \d x.
	\end{array}
	\right.
\end{align*}

The case of equal dispersal rates, $\delta = 0$, has received a lot of attention over the years. The system was first proposed in \cite{GP84} as a model for two distinct interacting species moving in such a way that they avoid overcrowding. In  \cite{BGHP85, BGH87, BGH87a} the problem was studied as a free boundary problem for ordered initial data, i.e., $\supp(\rho) < \supp(\eta)$ or vice versa, in order to deal with possible gaps between both supports.  Only a couple of decades later, reaction terms were added to the system and an existence theory was developed (see \cite{BHIM12}) more general initial data were allowed (see \cite{Carrillo2017}) and the restriction to one spatial dimension was removed (see \cite{Gwiazda2018}). It is important to highlight that the strategies employed in these works differ strongly from each other.

\subsection{Existence and Segregation for $\delta < 0$}

Since the segregation result of this section relies on the theory of optimal transportation,   we shall briefly recapitulate the most important notions for the sake of a rigorous presentation.   Then the idea is to think of $\rho$ and $\eta$ as elements of the set of probability measures, $\curlyP(\Omega)$. For the purpose of introducing the relevant notations we shall switch the notation to general measures $\mu, \nu$ rather than $\rho, \eta$ lest the reader be confused.

Let us begin by introducing the notion of push-forward measures. To this end let $\mu \in \curlyP(\Omega)$ and $\curlyT:\Omega \rightarrow \Omega$ be a Borel measurable function. We denote by $\curlyT_\#\mu \in \curlyP(\Omega)$ the push-forward of $\mu$ by the transport map $\curlyT$, defined by
\begin{align*}
\int_\Omega f \d \curlyT_\# \mu =\int_\Omega f\circ \curlyT \d \mu,
\end{align*}
for any Borel measurable functions $f:\Omega \rightarrow\R$. Next, let us note that for any two probability measure $\mu, \nu \in \curlyP(\Omega)$
\begin{align*}
	d_2(\mu, \nu) := \left(\inf_{\gamma \in \Gamma(\mu, \nu)} \iint_{\Omega \times \Omega} |x-y|^2 \d \gamma \right)^{1/2},
\end{align*}
defines a metric on $\curlyP(\Omega)$ which is usually referred to as 2-Wasserstein distance. Here $\Gamma(\mu, \nu)$ stands for the set of transport plans with marginals $\mu$ and $\nu$, that is 
\begin{align*}
	\Gamma(\mu, \nu) := \left\{ \gamma \in \curlyP(\Omega\times\Omega)\;|\; \pi_\#^1\gamma = \mu, \text{ and } \pi_\#^2\gamma = \nu\right\},
\end{align*}
where $\pi^i:\Omega\times\Omega \rightarrow \Omega$ denotes the projection onto the $i$th component of the product space. Moreover, we call $\Gamma_o(\mu, \nu)$ the set of optimal transport plans and note that
\begin{align*}
	d_2^2(\mu, \nu) = \iint_{\Omega \times \Omega} |x-y|^2 \d \gamma,
\end{align*}
whenever $\gamma\in\Gamma_o(\mu, \nu)$, a fact we shall exploit later a lot.    Endowed with the 2-Wasserstein metric, the space $(\curlyP(\Omega), d_2)$ is a complete metric space. Finally, the link between transport plans and transport maps can be established as follows. Whenever $\mu \ll \curlyL$, i.e. $\mu$ is absolutely continuous with respect to the Lebesgue measure, there exists a unique transport plan $\gamma$ that can be written as $\gamma= (\id, \curlyT)$ where $\curlyT$ pushes $\mu$ onto $\nu$, $\curlyT_\#\mu = \nu$.

With these notions at hand, let us recall the Jordan-Kinderlehrer-Otto (JKO) scheme; cf. \cite{JKO98}.  For a given initial datum $\mu^0\in \curlyP(\Omega)$ one recursively defines a sequence

\begin{align}
	\label{eq:JKOproblem}
	\mu_\tau^{n+1} \in \operatorname{argmin}_{\mu \in \curlyP(\Omega)} \left\{\frac{1}{2\tau} d_2^2(\mu,\mu^n) + \Len(\mu)\right\},
\end{align}
with $\Len$ as defined in \eqref{eq:local_energy}.

\begin{remark}\label{rem:noGradFlow}
	In Proposition \ref{prop:nonlsc} we have seen that the local energy is not lower semicontinuous with respect to the weak convergence. Thus the theory of \cite{AGS08} can not be applied, and optimisers  to the minimising movement scheme \eqref{eq:JKOproblem} need to be constructed directly. In particular, we stress that $\Len$ does not give rise to a 2-Wasserstein gradient flow due to the lack of lower semicontinuity in the energy.
\end{remark}

We shall see that, in our case, problem  \eqref{eq:JKOproblem} is however well-defined and admits a sequence of minimisers. The convergence result below is a consequence of the theory developed in \cite{CFSS17}.
\begin{theorem}[Convergence to weak solutions] Let $\rho_0,\eta_0\in BV(\Omega)$ be segregated and such that there exists a BV-function $0\leq r_0 \leq 1$ such that $r_0 = \rho_0/(\rho_0+\eta_0)$ on $\{\rho_0+\eta_0>0\}$. Then, problem \eqref{eq:JKOproblem} is well-posed and admits a  sequence $(U_\tau^n)_{n\in\N}$ for any $\tau\in(0,1)$. Moreover, the piecewise constant interpolations $(\rho_\tau)_{\tau>0}$ and $(\eta_\tau)_{\tau>0}$  associated to the sequences converge to a weak solution of
\begin{align}
	\label{eq:Bertsch}
	\left\{
	\begin{array}{rl}
	\begin{split}
	\!\!\!\! \partialt \rho &= (1+\delta)\px \left(\rho \px \left(\rho + \eta \right)\right),\\[1em]
	\!\!\!\! \partialt \eta &= (1+\delta)\px \left(\eta \px \left(\rho + \eta\right)\right),
	\end{split}
	\end{array}
	\right.
\end{align}
in the sense of \cite[Def. 2.6]{CFSS17}, and the solution $(\rho, \eta)$ remains segregated for all time.
\end{theorem}
\begin{proof}
The proof is based on the splitting strategy in \cite{CFSS17}. In fact, the theory can be applied directly under the above assumptions on the initial data whenever $0<\tau<1$ and by choosing the internal energy density $\chi(x) = x^2/2$ which corresponds to the functional
$$
	\curlyE(\rho, \eta) =  \frac{1+\delta}{2}\int_\Omega(\rho+\eta)^2\d x,
$$
in our case. As for the construction of minimisers, let the previous iteration, $U_\tau^n = (\rho_\tau^n, \eta_\tau^n)$, be given and assume that $U_\tau^n$ is segregated, i.e., $\rho_\tau^n \eta_\tau^n = 0$, almost everywhere. Let
	\begin{align}
		\label{eq:reduced_functional}
		U^\star = (\rho^\star, \eta^\star) \in \operatorname{argmin}\left\{\frac{1}{2\tau} d_2^2(U,U^n) + \frac{1+\delta}{2} \int_{\Omega} (U_1  + U_2)^2 \d x,
		\right\}
	\end{align}
	be given. Note that such a minimiser exists and has the same optimality condition as the problem corresponding to the porous medium equation for one density; see \cite[Lemma 3.3]{CFSS17}. Moreover, by \cite[Theorem 3.3]{CFSS17}, we know that $U^\star$ is segregated, i.e., $\rho^\star \eta^\star = 0$, a.e. in $\Omega$. 
	
	We claim that $U^\star$ is also a minimiser for \eqref{eq:JKOproblem}. Assume it is not and let $\tilde U$ be a competitor with strictly lower energy. Then, 
	\begin{align*}
		0&\leq \frac{1}{2\tau} d_2^2(\tilde U,U^n) + \frac{1+\delta}{2} \int_{\Omega} (\tilde \rho  + \tilde \eta)^2\d x - \delta \int_\Omega \tilde \rho \tilde \eta\d x\\
		&<\frac{1}{2\tau} d_2^2( U^\star,U^n) + \frac{1+\delta}{2} \int_{\Omega} ( \rho^\star  + \eta^\star)^2 \d x - \delta \int_\Omega  \rho^\star \eta^\star\d x.
	\end{align*}
	Using the fact that $U^\star$ is segregated, we can rearrange this inequality to
	\begin{align*}
		0 &\leq  - \delta \int_\Omega \tilde \rho \tilde \eta \d x\\
		&< \frac{1}{2\tau} d_2^2( U^\star,U^n) + \frac{1+\delta}{2} \int_{\Omega} ( \rho^\star  + \eta^\star)^2 \d x \\
		&\phantom{\leq}-\left( \frac{1}{2\tau} d_2^2(\tilde U,U^n) + \frac{1+\delta}{2} \int_{\Omega} (\tilde \rho  + \tilde \eta)^2 \d x \right).
	\end{align*}
	Recalling that $U^\star$ is optimal, cf. Eq. \eqref{eq:reduced_functional}, we reach a contradiction since then $0<0$, and we deduce that $\tilde U$ is not a feasible competitor. Hence, $U^\star$ is not only a minimiser of \eqref{eq:reduced_functional} but also of the original problem \eqref{eq:JKOproblem}. Setting $U^{n+1}=(\rho_\tau^{n+1},\eta_\tau^{n+1}):= U^\star$ concludes the first statement of the lemma. Defining the piecewise constant interpolation
	\begin{align*}
		U_\tau(t,x) := U_\tau^n(x),
	\end{align*}
	for all $(t,x) \in \big[n\tau, (n+1)\tau\big)\times\Omega$ and $(\rho_\tau, \eta_\tau):=U_\tau$. Then, as $\tau \rightarrow 0$, $(\rho_\tau, \eta_\tau)$ converges to a weak solution of \eqref{eq:Bertsch}; see \cite[Theorem 2.9]{CFSS17} and \cite[Lemma 3.14]{CFSS17}.
\end{proof}

\begin{remark}\label{rem:JKOdynamiceqn}
	This result is quite remarkable taking into account the fact that we set out to construct a JKO sequence for 
\begin{align*}
	\pt \rho &= \px \left(\rho \px \left((1+\delta)\rho + \eta \right)\right),\\
	\pt \eta &= \px \left(\eta \px \left(\rho + (1+\delta)\eta\right)\right),
\end{align*}
with $\delta \in (-1,0)$. On the level of the JKO we preserve segregation which prevents terms of the form $\rho \nabla \eta$ and $\eta \nabla \rho$ from appearing in the weak formulation. In the introduction we remarked that the determinant of the mobility has a negative sign only if both species overlap. In a way, constructing an approximation using the minimising movement scheme avoids the  backward diffusion regimes already on the level of the approximations, and, as a consequence, the cross-terms do not appear in the weak formulation. By Remark \ref{rem:noGradFlow}, the solution is not a gradient flow for the functional $\Len$, but for the relaxed functional
\begin{align*}
	\curlyE(\rho, \eta):= \frac{1+\delta}{2}\int_\Omega (\rho + \eta)^2 \d x
\end{align*} 
 An illustration of this behaviour can be seen in the numerical section; cf. Figure \ref{fig:deltaplusone}.
\end{remark}

    
\section{Formation of Gaps in the Non-Local Case}\label{sec:gapstudy}
While Lemma \ref{lem:segregated_minimisers_delta_neg} ensures segregations of minimisers for the nonlocal energy $\NLen$ with $\delta \in (-1,0)$ already, this effect can be more pronounced in the sense that there exists a positive distance between the supports of the two species. 
\subsection{Necessary conditions of the gap size of minimisers}

\begin{theorem}[Necessary conditions of the gap width] Let $\delta \in (-1, 0)$ be fixed and $K$ be a given kernel with  compact support on $[-\alpha,\alpha]$ for $\alpha > 0$. 
Then any minimiser of $\NLen$  satisfies the inequality
\begin{align*}
 (1+\delta)(m_1+m_2)(\rho + \eta) - \delta(m_1 K\star \eta + m_2 K\star \rho) \ge 2 \NLen(\rho,\eta),
\end{align*}
almost everywhere in $\Omega$.
In particular, we necessarily have $r \le \alpha$ whenever there exists a gap of with $2r$. 
\end{theorem}
\begin{proof} The strategy of the proof is to consider special variations of the energy as in Theorem \ref{thm:exminlocal}. To this end, fix $\psi \in C^\infty_c(\Omega)$ such that $\psi \geq 0$. Recall that $m_1$ and $m_2$ denote the respective masses of $\rho$ and $\eta$ as before. For given minimisers $\rho$ and $\eta$ we consider the following perturbation 
\begin{align*}
\rho_\eps = \rho + \eps \nu\quad\text{ with }\quad \nu = m_1\psi - \rho \int_\Omega \psi \d x,
\end{align*}
for $0<\epsilon<\|\psi\|_{L^1(\Omega)}^{-1}$. This choice of variations can be easily checked to ensure  that the perturbations, $\rho_\epsilon$, remain in the set of nonnegative measures with given mass. 
by the restriction on $\epsilon$. Calculating the variation of $\NLen$ we obtain
\begin{align}
\label{eq:variation}
	\begin{split}
		&\lim_{\eps \to 0}\frac{\NLen(\rho_\eps,\eta)- \NLen(\rho,\eta)}{\eps}\\
		&\quad = \int_\Omega \psi\left[m_1(1+\delta)(\rho+\eta) - m_1\delta K\star \eta - \int_\Omega (1+\delta)\rho(\rho+\eta) - \delta \rho K\star \eta \d y\right] \d x \ge 0,
	\end{split}
\end{align}
where we have the inequality due to the positivity of the energy functional and the fact that $(\rho,\eta)$ is a minimiser. Using $\psi \ge 0$ in conjunction with the fact that $\rho$ and $\eta$ are segregated, cf. Lemma \ref{lem:segregated_minimisers_delta_neg}, we obtain
 \begin{align*}
 m_1(1+\delta)(\rho+\eta) - m_1\delta K\star \eta - (1+\delta)\int_\Omega \rho^2 \d x + \delta \int_\Omega\rho K\star \eta \d x \ge 0,
\end{align*}
for almost every $x\in\Omega$. Adding a perturbation to $\eta$ instead of $\rho$ yields
\begin{align*}
 m_2(1+\delta)(\rho+\eta) - m_2\delta K\star \rho - (1+\delta)\int_\Omega \eta^2\d x + \delta \int_\Omega\eta K\star \rho \d x \ge 0,
\end{align*}
almost everywhere in $\Omega$.
Adding the equations above and using the definition of $\NLen$ yields
\begin{align}\label{eq:gap_contradiction}
 (1+\delta)(m_1+m_2)(\rho + \eta) - \delta(m_1 K\star \eta + m_2 K\star \rho) \geq 2 \NLen(\rho,\eta).
\end{align}
Now assume that there exists indeed a gap given as 
$$
A:= (-r,r) \subset \Omega \setminus \left( \supp(\rho) \cup \supp(\eta)\right),
$$
for some positive $r>0$. Since $K$ is assumed to be compactly supported on the interval $[-\alpha, \alpha]$, we know that $K\star \rho$ and $K\star \eta$ are supported in $(-r+\alpha,r-\alpha)$. 

If $2r$ were to be strictly greater that $2\alpha$, we could fix a nonempty interval in the gap, on which all terms on the left-hand side of \eqref{eq:gap_contradiction} are zero. This yields a contradiction as we know that $\NLen$ is positive.
\end{proof}

\subsection{Explicit examples for special potentials}
We will examine this behaviour for two special choices of the kernel function $K$ and study the energy
\begin{align*}
	\NLen(\rho, \eta) =\frac{1+\delta}{2}\int(\rho + \eta)^2 \d x - \delta\int \rho K\star \eta \d x,
\end{align*}
for functions $\rho, \eta\in L_+^1([-L,L])$.   Using the same perturbations for $\rho$ and $\eta$ as in the proof of Theorem \ref{thm:exminlocal}, part (2), we get
\begin{align*}
	(1+\delta) (\rho + \eta) - \delta K\star \eta = c_1, \quad \text{ on } \supp (\rho).
\end{align*}
as well as
\begin{align*}
	(1+\delta) (\rho + \eta) - \delta K\star \rho = c_2, \quad \text{ on } \supp (\eta),
\end{align*}
where
\begin{align*}
	\left\{
	\begin{array}{l}
	c_1 = \ds \frac{1}{m_1}\int_\Omega (1+\delta)\rho (\rho + \eta) -\delta \rho K\star\eta\d y,\\[1em]
	c_2 = \ds \frac{1}{m_2}\int_\Omega (1+\delta)\eta (\rho + \eta) -\delta \rho K\star\eta\d y.
	\end{array}
	\right.
\end{align*}
By Lemma \ref{lem:segregated_minimisers_delta_neg}, we already know that minimisers are segregated. Since we are looking for critical points with a positive gap, we assume segregation as well and obtain
\begin{align}
	\label{eqn:necessary_condition_minimiser}
	\left\{
	\begin{array}{rl}
		\ds (1+\delta) \rho - \delta K\star \eta &= c_1\text{ on } \supp (\rho),\\[0.2em]
		\ds (1+\delta) \eta - \delta K\star \rho &= c_2\text{ on } \supp (\eta).
	\end{array}
	\right.
\end{align}
For the rest of this section we consider the symmetric case, i.e., $\eta(x) = \rho(-x)$ for two kernels $K \in \{{\tilde K}, {\bar K}\}$, where
\begin{align}\label{eq:Kpicard}
	{\tilde K}(x) = \frac{1}{2\alpha} \exp\left(-\frac{|x|}{\alpha}\right),
\end{align}
and
\begin{align*}
{\bar K}(x) = \frac{1}{2\alpha}\indicator_{[-\alpha, \alpha]}(x).	
\end{align*}
Note that these kernels are substantially different in the sense that $\supp({\tilde K})=\R$, whereas $\supp({\bar K})=[-\alpha, \alpha]$.
In both cases, we can construct critical points of the energy that are strictly segregated by a gap whose width depends on the range of the kernel as well as the parameters $\alpha$, $\delta\in(-1,0)$ and $L$. At least in the case of the indicator kernel, we are also able to show that these critical points, in the sense of \eqref{eqn:necessary_condition_minimiser}, are indeed minimisers. Our results are summarised as follows.
\begin{proposition}\label{thm:gaps} Given $\Omega = [-L,L]$, let us denote by $(\rho,\eta)$ critical points to $\NLen$, i.e. solutions to \eqref{eqn:necessary_condition_minimiser}, with either $K={\tilde K}$ or $K= {\bar K}$. Then there exist critical points $(\rho,\eta)$ with $\rho(x) = \eta(-x)$ and a positive constant $r \ge 0$ such that 
\begin{align*}
 \supp(\rho) \subset [-L, -r] \text{ and } \supp(\eta) \subset [r, L].
\end{align*}
In particular:
\begin{itemize}
 \item For $K={\tilde K}$ one such critical point is given by
\begin{align*}
 \rho(x) = c_\delta \left( e^{\frac{x+r}{\alpha}} - 1\right),
\end{align*}
with 
\begin{align*}
 c_\delta = \left(\alpha \left(1 - \exp\left(-\frac{L-r}{\alpha}\right)\right)- (L-r)\right)^{-1}
\end{align*}
and 
\begin{align*}
r = L + \alpha\log\left(\frac{\delta+2\sqrt{-(\exp(L/\alpha))^2\delta(1+\delta)}}{(4+4\delta)\exp(2L/\alpha)+\delta)}\right)
\end{align*}
\item For $K= {\bar K}$ one critical point is of the form
\begin{align*}
\rho(x) = b\left[ \cos\left(\lambda x + \frac12 \alpha \lambda\right) + \sin\left(\lambda x + \frac12 \alpha \lambda\right)\right],
\end{align*}
with 
\begin{align*}
 b = h \left[\cos\left(\lambda r - \frac12 \alpha \lambda\right) +\sin\left(\lambda r - \frac12 \alpha \lambda\right)\right]^{-1},
\end{align*}
\begin{align*}
 h = \left(L + r - \alpha - \frac{2}{\lambda}\frac{1}{1 + \cot\left(\lambda r - \frac12 \alpha \lambda\right)}\right)^{-1},\quad \lambda = \frac{\delta}{1+\delta} \frac{1}{2\alpha}, 
\end{align*}
and
\begin{align}\label{eq:rindicator}
 r = \frac\alpha2\left(1  + \frac{1+\delta}{\delta}\pi\right).
\end{align}
In this case, the critical point is in fact a minimiser in the class of segregated densities.
\end{itemize}
\end{proposition}

\begin{remark} 
\label{rem:criticaldelta}
In the case of the indicator kernel, equation \eqref{eq:rindicator} implies that, independent of the value of $\alpha$, positive gaps only appear for 
\begin{align}
	\label{eq:criticaldelta}
	\delta < \delta_{{\rm critical}} = -\frac{\pi}{1+\pi}\approx -0.7585.	
\end{align}
In this case, the size of the gap grows linearly in $\alpha$. For the Picard kernel, there also exists a threshold which will however depend on $L$ and $\alpha$ and the size of the gap depends non-linearly on $\alpha$ as well. These findings are confirmed by numerical simulations in Figures \ref{fig:gap_indicator} and \ref{fig:gap_picard} in Section \ref{sec:gapstudy}.
\end{remark}

\begin{proof}
The proof is based on the fact that ${\tilde K}$ and ${\bar K}$ are fundamental solutions to differential equations. Indeed, ${\tilde K}$ is a solution for the operator $\id - \alpha^2 \partial_x^2$, i.e.,
\begin{align*}
	\left(\id - \alpha^2 \frac{\partial^2}{\partial x^2}\right) {\tilde K} = \delta_0,
\end{align*}
where $\delta_0$ denotes the Dirac distribution at $x=0$. For ${\bar K}$ we have 
\begin{align*}
	\frac{\partial}{\partial x} K(x) = \frac{1}{2\alpha}\left(\delta_{-\alpha} - \delta_{\alpha}\right).
\end{align*}
Imposing the additional boundary condition $\rho(-r) = 0$ (and thus $\eta(r) = 0$), we obtain, after some calculations, the explicit forms above. In order to show that for $K=\bar K$ the critical point is indeed a minimiser in the class of all functions segregated with gap $r$, we proceed as follows: Denote by $u,\,v$ perturbations such that
\begin{align*}
 \supp(u) = [-L,-r],\; \supp(v) = [r,L]\quad \text{ and } \quad\int_{-L}^{-r} u\d x = \int_r^L v\d x = 0.
\end{align*}
Then we have, with $\rho$ and $\eta$ given as in the statement of Proposition \ref{thm:gaps},
\begin{align*}
 \NLen(\rho+u,\eta + v) &=  \NLen(\rho,\eta) + \frac{1+\delta}{2}\int_{-L}^L 2(\rho+\eta)(u+v) + (u+v)^2\d x\\
 &-\delta \int_{-L}^L v \bar K\star \rho + u \bar K\star \eta + u \bar K\star v\d x.
\end{align*}
Rearranging terms and using \eqref{eqn:necessary_condition_minimiser}, we obtain
\begin{align*}
	\NLen(\rho+u,\eta + v) &=  \NLen(\rho,\eta) + \frac{1+\delta}{2}\int_{-L}^L (u+v)^2\d x - \delta \int_{-L}^L u\bar K\star v\d x\\
&+ c_1\int_{-L}^L v\d x + + c_2\int_{-L}^L u\d x.
\end{align*}
As the perturbations $u$ and $v$ have zero mass and as their supports are disjoint with $\eta$ and $\rho$ respectively, all terms in the second line above are zero. In remains to examine the convolution term. Using the definitions of $\bar K$, we have
\begin{align*}
 -\delta \int_{-L}^L u \bar K\star v\d x = -\delta \int_{-L}^{-r} u(x) \int_{-\alpha}^\alpha v(x-y)\d y \d x.
\end{align*}
However, from the explicit for of $r$ given in \eqref{eq:rindicator}, we see that we always have $\alpha < 2r$, so that in the second integral above, $x-y$ will never lie in the support of $v$ and thus, the whole integral is zero. We arrive at 
\begin{align*}
	\NLen(\rho+u,\eta + v) &=  \NLen(\rho,\eta) + \frac{1+\delta}{2}\int_{-L}^L (u+v)^2\d x,
\end{align*}
which shows that the critical point given in in the statement of Proposition \ref{thm:gaps} is indeed a minimiser.
\end{proof}


\section{Numerical Study}\label{sec:numerics}
This section is dedicated to an extensive numerical study to confirm the results obtained in the previous section. To solve the time-dependent problem we use the finite volume scheme recently introduced in \cite{CFS18}. This scheme is particularly suited for our purposes since it preserves a discrete energy inequality and was able to accurately reproduce segregated solutions. In order to numerically calculate minimisers of the energies $\Len$ and $\NLen$ we use a projected steepest descent scheme applied to the Lagrangian that consists of the respective functional plus appropriate Lagrange multipliers to ensure the mass constraint. In each step, a projection is performed to ensure that the densities remain non-negative.

\subsection{Dynamical Behaviour}\label{sec:dynamical}

\subsubsection{Comparison of $\delta\in(-1,0)$ and $\delta=0$ with Reduced Diffusion Coefficient in the Local System}
We start our study by examining the dynamical behaviour of solutions. In particular, we compare solutions to the full system to \eqref{eq:system} with those of 
\begin{align}
\label{eq:deltazero_diffusioncoefficient}	
\begin{split}
	\pt \rho &= (1+\delta)\px\left(\rho \px(\rho + \eta))\right),\\
	\pt \eta &= (1+\delta)\px\left(\eta \px(\rho + \eta))\right),
\end{split}
\end{align}
i.e., the $\delta = 0$ version of \eqref{eq:system} just with an reduced overall diffusion coefficient of $(1+\delta)$. This is closely related to the analytic problems related to the terms $ \rho\nabla \eta \text{ and } \eta\nabla\rho$, as discussed in the introduction.
Indeed, given the results of Remark \ref{rem:JKOdynamiceqn}, we expect that weak solutions of \eqref{eq:deltazero_diffusioncoefficient} are also solutions of the original system \eqref{eq:system}. As least numerically, this in the case as demonstrated in Figure \ref{fig:deltaplusone}.

\begin{figure}[ht!]
\includegraphics[width=0.75\textwidth]{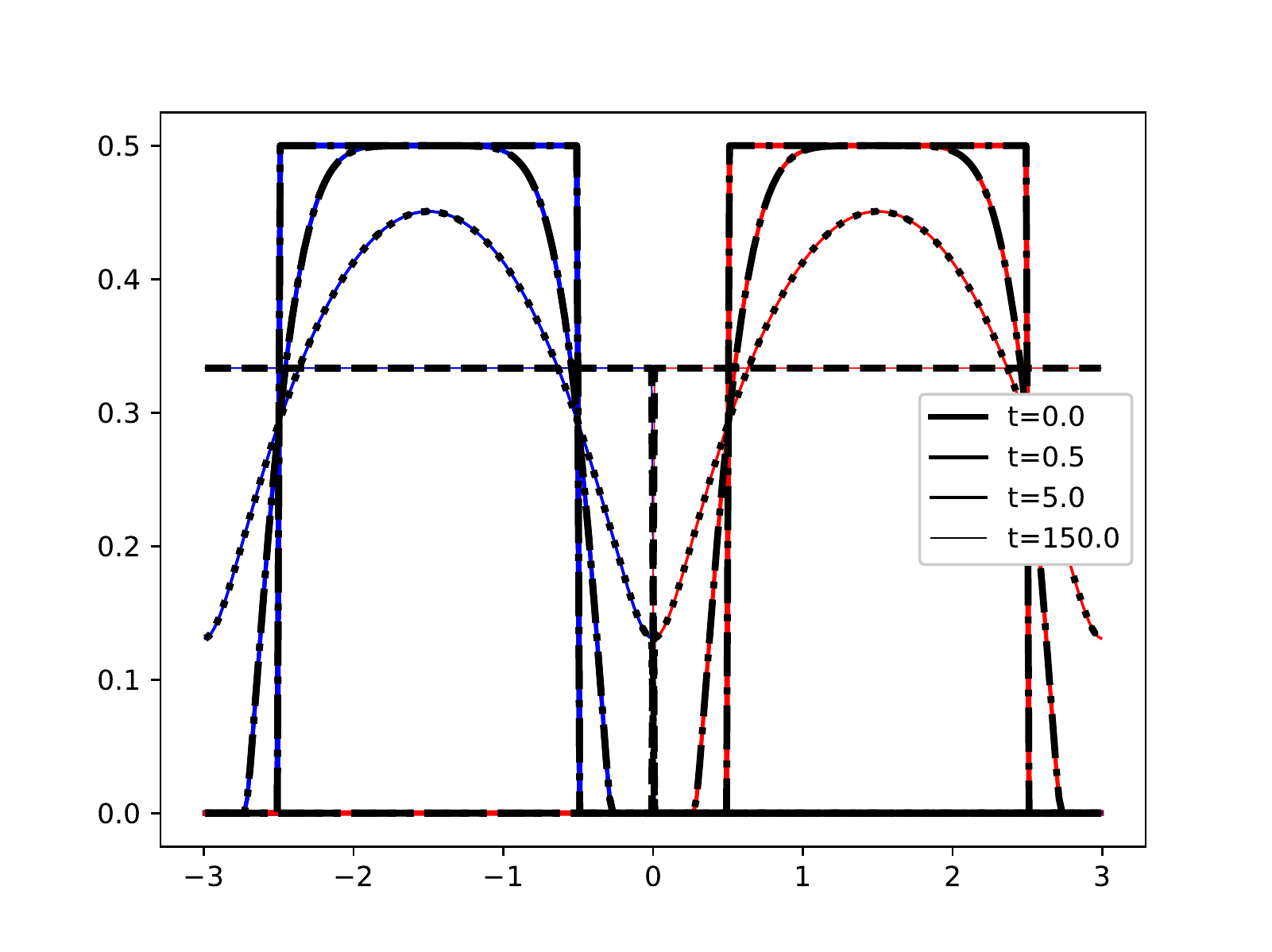}
	\caption{The blue and the red curves correspond to the solution of system \eqref{eq:system} for the local functional $\Len$. The dotted and dashed lines correspond to the solution of \eqref{eq:deltazero_diffusioncoefficient}. In both cases we chose $\delta = -0.9$.}
	\label{fig:deltaplusone}
\end{figure}

\subsubsection{The Local System in Different Regimes of $\delta$}

\begin{figure}[ht!]
	\centering
	\subfigure[$\delta=-0.9$.]{
		\includegraphics[width=0.3\textwidth]{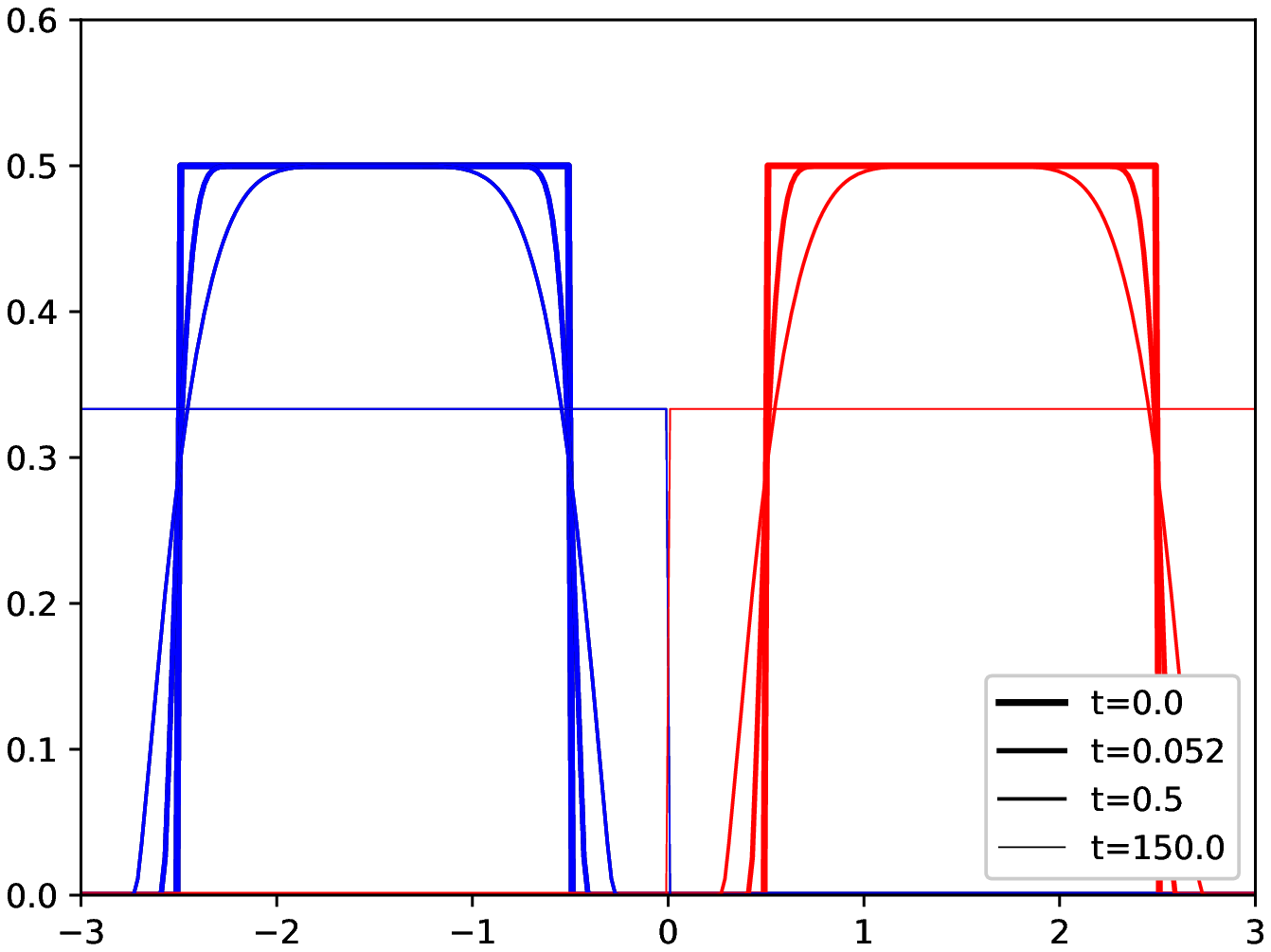}
	}
	\subfigure[$\delta=0$.]{
		\includegraphics[width=0.3\textwidth]{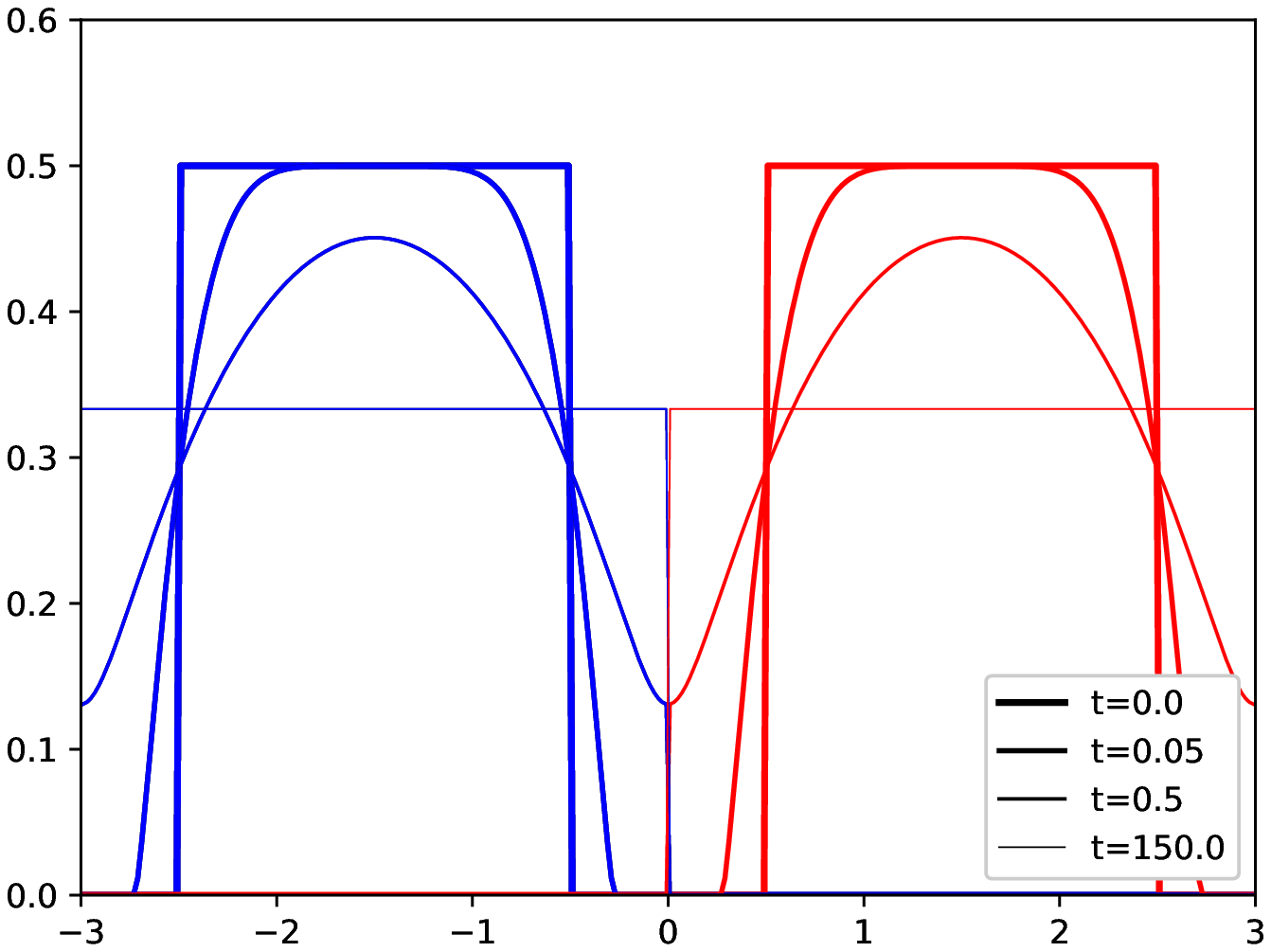}
	}
	\subfigure[$\delta=0.9$.]{
		\includegraphics[width=0.3\textwidth]{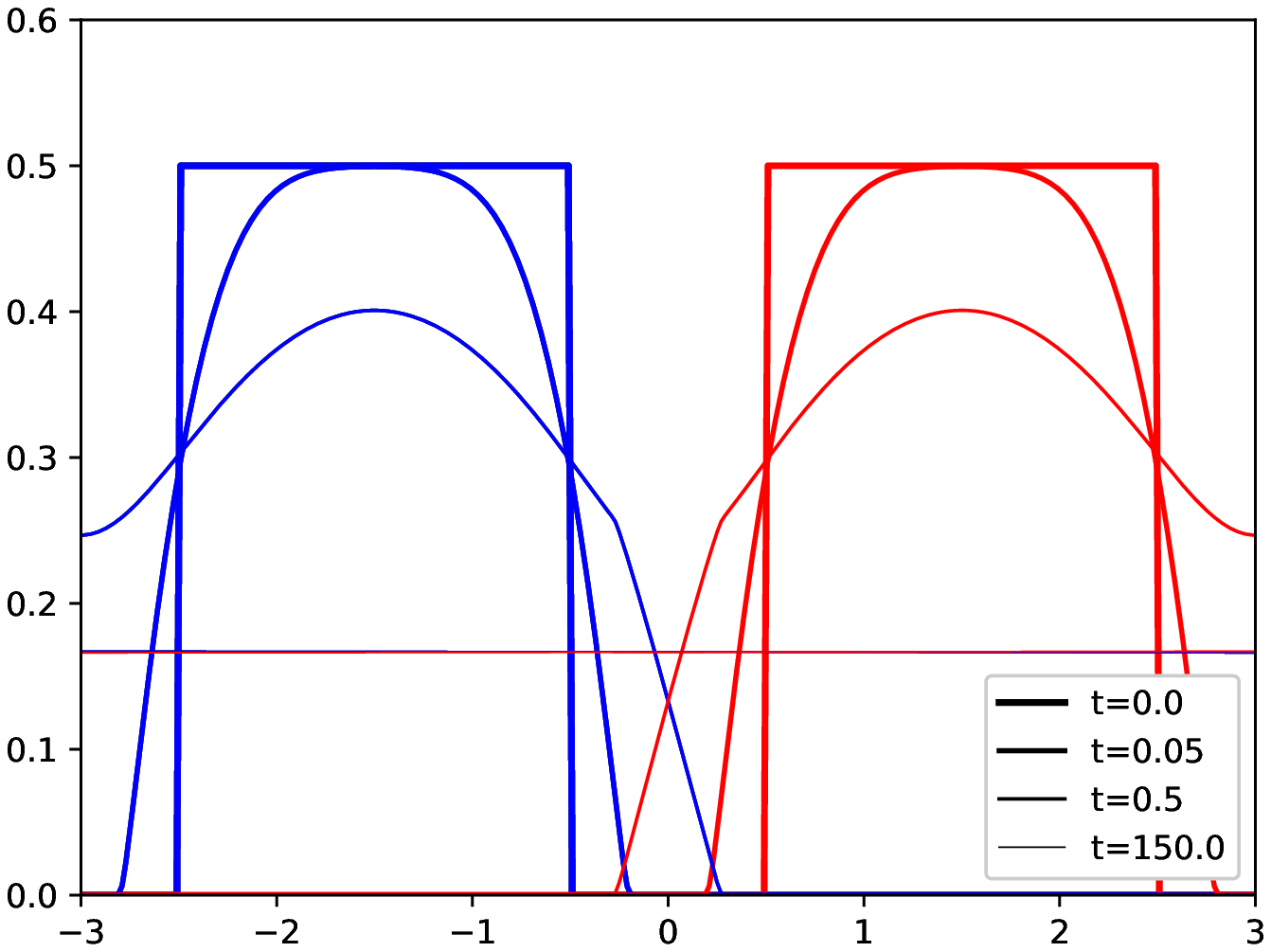}
	}
	\caption{Starting with the same initially segregated initial data $\rho=\indicator_{[-2.5,\,-0.5]}$ and $\eta=\indicator_{[0.5,\,2.5]}$ we study the evolution of the local system \eqref{eq:system} corresponding to $\curlyF = \Len$ for $\delta$ in the three different regimes.}
	\label{fig:PDElocal_deltas}
\end{figure}
Next we consider the behaviour of the local system, i.e., \eqref{eq:system} with $\curlyF = \Len$ for different values of $\delta$ in the interval $(-1,0)$. The outcome is depicted in Figure \ref{fig:PDElocal_deltas}. As expected from our analytical results for the local energy in Theorem \ref{thm:exminlocal}, we observe that for long times and $\delta \in (-1,0)$, we obtain segregated states whose sum is constant and which fill the whole domain. For $\delta >0$ on the other hand, we observe mixing of $\rho$ and $\eta$ and, for long times, convergence to constant stationary states determined by the size of the domain and the mass of the initial data. Only in the case $\delta = 0$ both mixing and segregation are possible and the behaviour is dictated by the initial data. In \cite{CFSS17}, the authors proved that segregated data remain segregated. This behaviour is also seen in Figure \ref{fig:PDElocal_deltas} (b). In addition, Figure \ref{fig:mixseg} displays the evolution for a mixed initial data. The numerically obtained stationary state consists of regions of coexistence as well as regions of single occupation.
\begin{figure}
	\centering
	\subfigure[$t=0$.]{
		\includegraphics[width=0.3\textwidth]{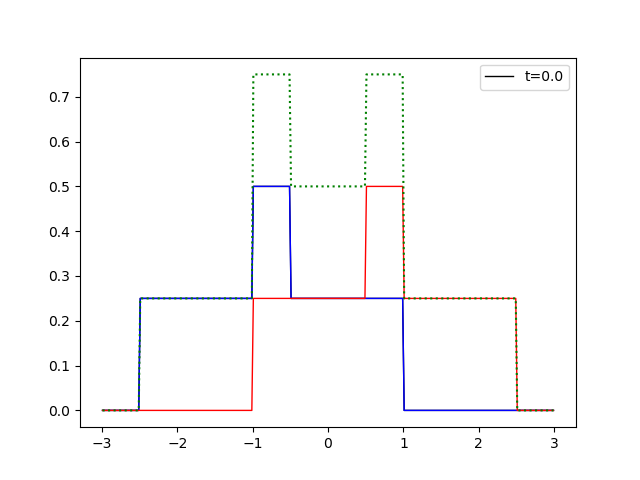}
	}
	\subfigure[$t=1$.]{
		\includegraphics[width=0.3\textwidth]{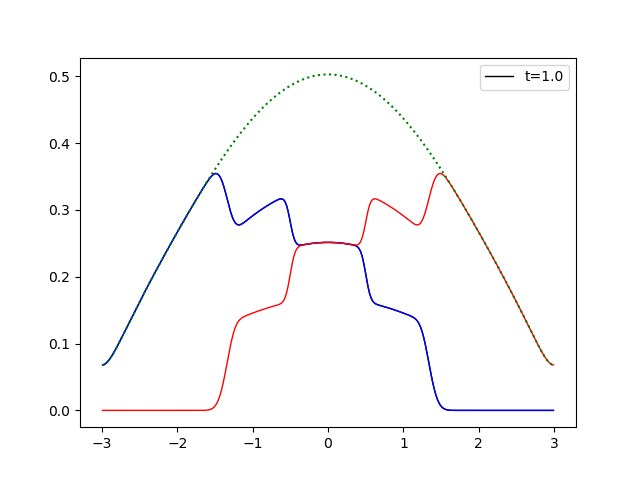}
	}
	\subfigure[$t=20$.]{
		\includegraphics[width=0.3\textwidth]{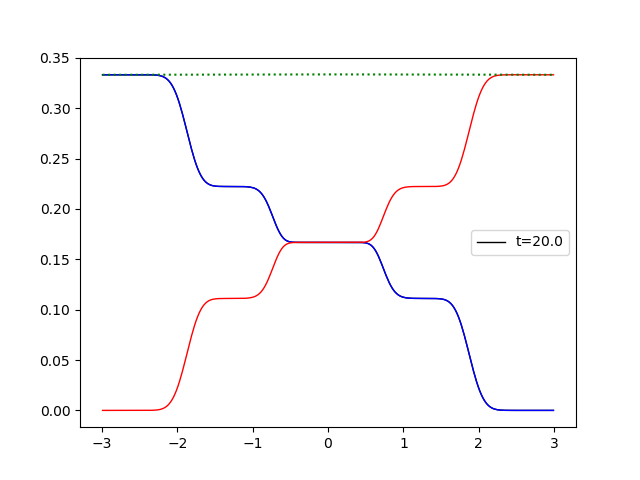}
	}
	\caption{The initial data $\rho_0(x) = \frac14\indicator_{[-2.5,-0.5]} + \frac14\indicator_{[-1, 1]}$ (blue) and $\eta_0(x)=\frac14\indicator_{[-1,1]} + \frac14\indicator_{[0.5, 2.5]}$ (red) are partially mixed. The sum (green) behaves like a Barenblatt profile and approaches a constant. Yet, the individual species remain mixed in the middle of the domain with regions of individual occupation on the sides of the domain.}
	\label{fig:mixseg}
\end{figure}

\subsubsection{The Non-local System in Different Regimes of $\delta$} Next we study the behaviour in the non-local case with the Picard kernel defined in \eqref{eq:Kpicard}, chosing $\alpha=2$. The results are presented in Figure \ref{fig:PDEnonlocal} and confirm that for $\delta = -0.9$ we observe the formation of gaps for long times.
\begin{figure}[ht!]
	\centering
	\subfigure[$\delta=-0.9$.]{
		\includegraphics[width=0.3\textwidth]{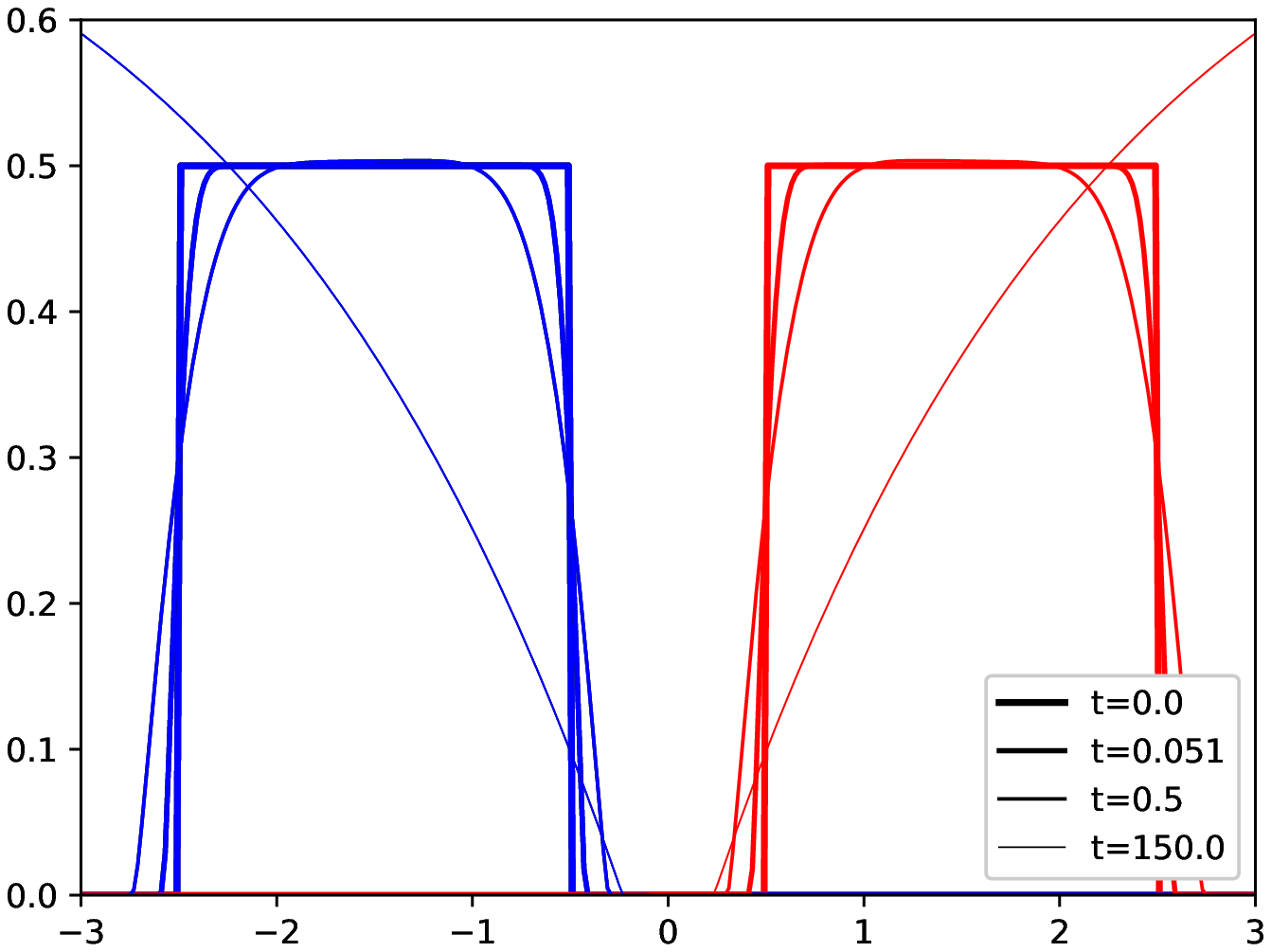}
	}
	\subfigure[$\delta=0$.]{
		\includegraphics[width=0.3\textwidth]{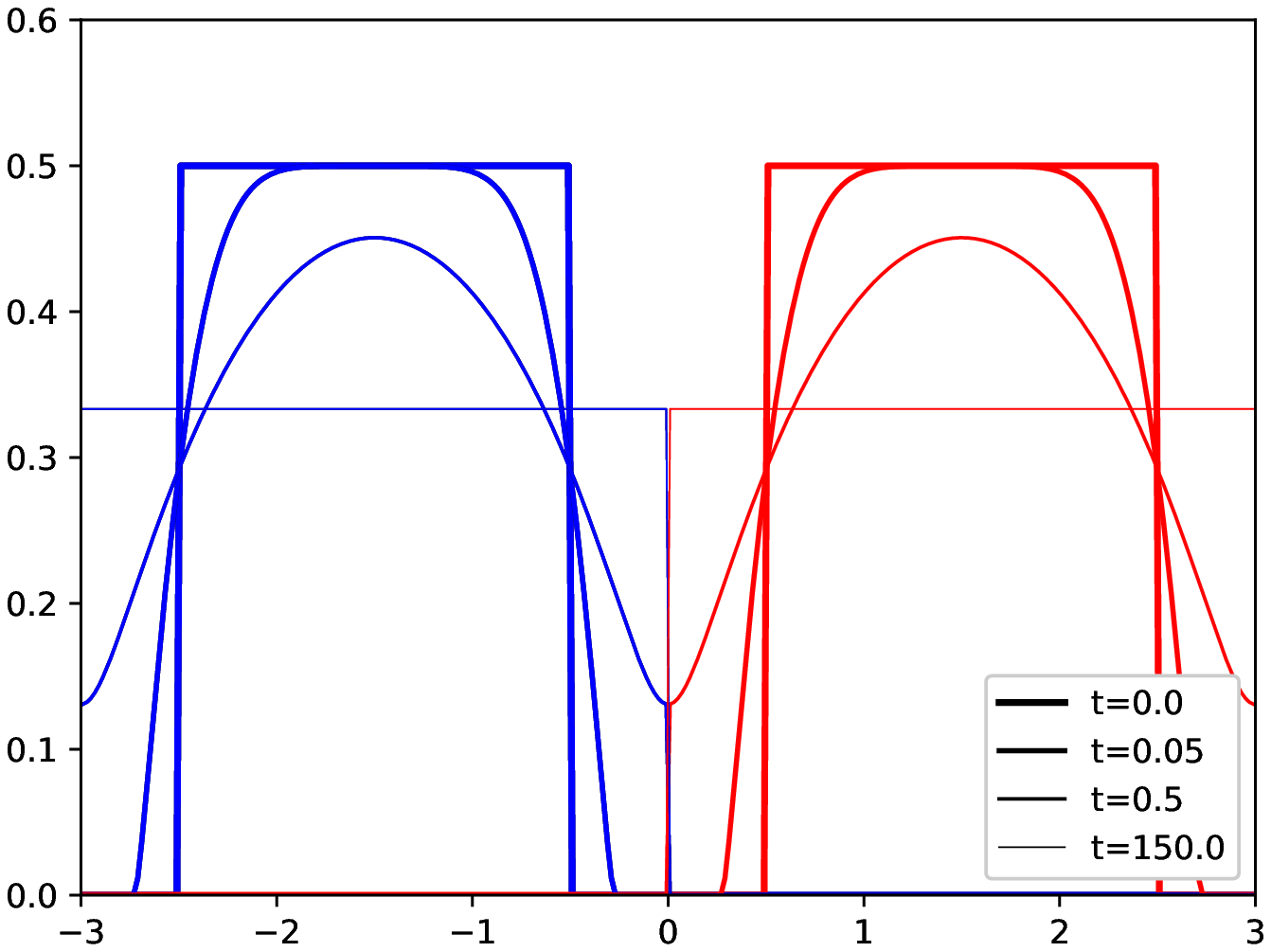}
	}
	\subfigure[$\delta=0.9$.]{
		\includegraphics[width=0.3\textwidth]{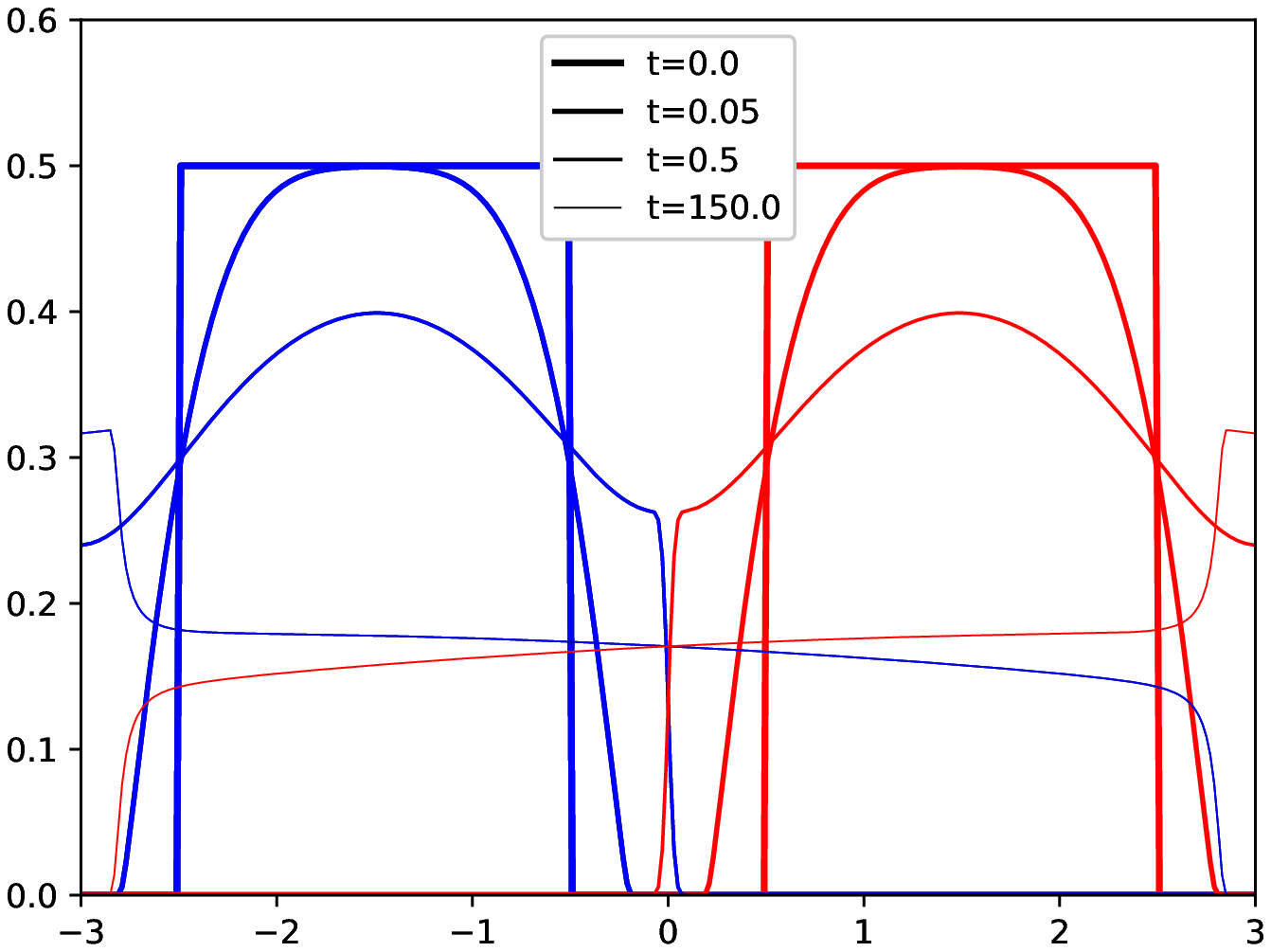}
	}
	\caption{Starting with the same initially segregated initial data $\rho=\indicator_{[-2.5,\,-0.5]}$ and $\eta=\indicator_{[0.5,\,2.5]}$ we study the evolution of the non-local system \eqref{eq:system} corresponding to $\curlyF = \NLen$ for $K=K_{{\rm Picard}}$, $\alpha = 2$, and for $\delta$ in the three different regimes.}
	\label{fig:PDEnonlocal}
\end{figure}

\subsection{Gap Study}
\label{sec:gapstudy}
The aim of this section is to show that the critical points presented in Theorem \ref{thm:gaps} can indeed be observed numerically, both as minimisers of the energy and as stationary solutions to the PDE system \eqref{eq:system}. In Figures \ref{fig:gap_indicator0}-\ref{fig:gap_indicator} and \ref{fig:gap_picard0}-\ref{fig:gap_picard}, we compare for different values of $\alpha$ and $\delta$ the explicit formulas to numerical minimisers and see that they are  indistinguishable and we study the gap width $r$ as a function of $\alpha$ and $\delta$ . In Figure \ref{fig:gap_PDE}, we have the same result for solutions of the PDE at $t=6000$ where we used approximated Dirac distributions, centred at $L=\pm 5$ as initial data.

\begin{figure}[ht!]
	\centering
	\subfigure[Indicator Kernel, $K=\bar K$.]
	{
		\includegraphics[width=0.45\textwidth]{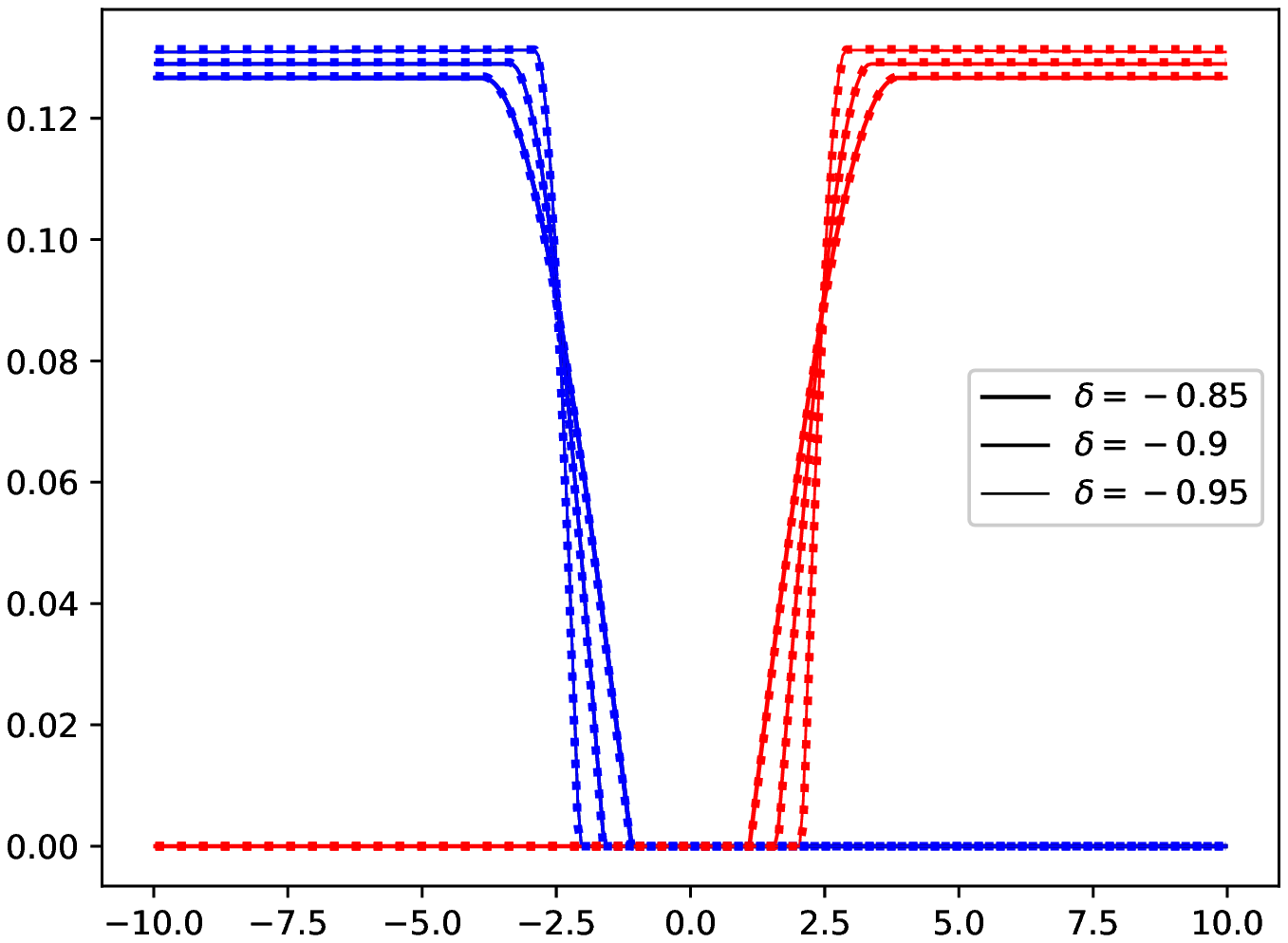}
	}
	\subfigure[Picard Kernel, $K = \tilde K$.]
	{
		\includegraphics[width=0.45\textwidth]{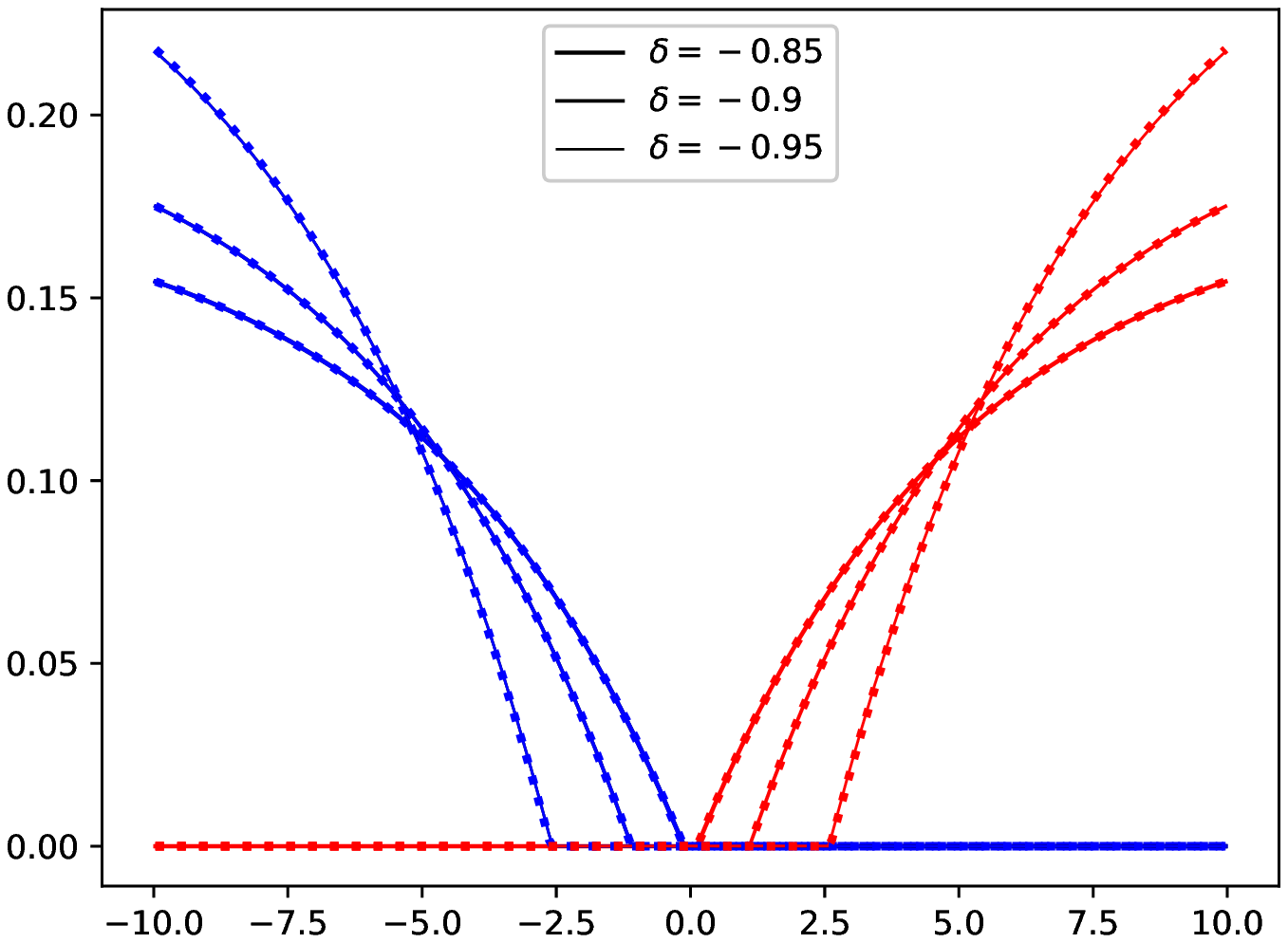}
	}
	\caption{Comparison of solution to the PDE system \eqref{eq:system} (solid lined) with the critical points computed in Section \ref{sec:gapstudy} (dotted lines), cf. Proposition \ref{thm:gaps}. In both cases ($K\in\{\tilde K, \bar K\}$) the system is initialised with $\rho_0(x) = \frac12 \indicator_{[-6,-4]}(x)$ and $\eta(x) = \frac12 \indicator_{[4,6]}(x)$, and we chose $\alpha = 5$.}
	\label{fig:gap_PDE}
\end{figure}

\begin{figure}[ht!]
	\centering
	\subfigure[Keeping $\delta=-0.9$ fixed while varying $\alpha$.]
	{
		\includegraphics[width=0.45\textwidth]{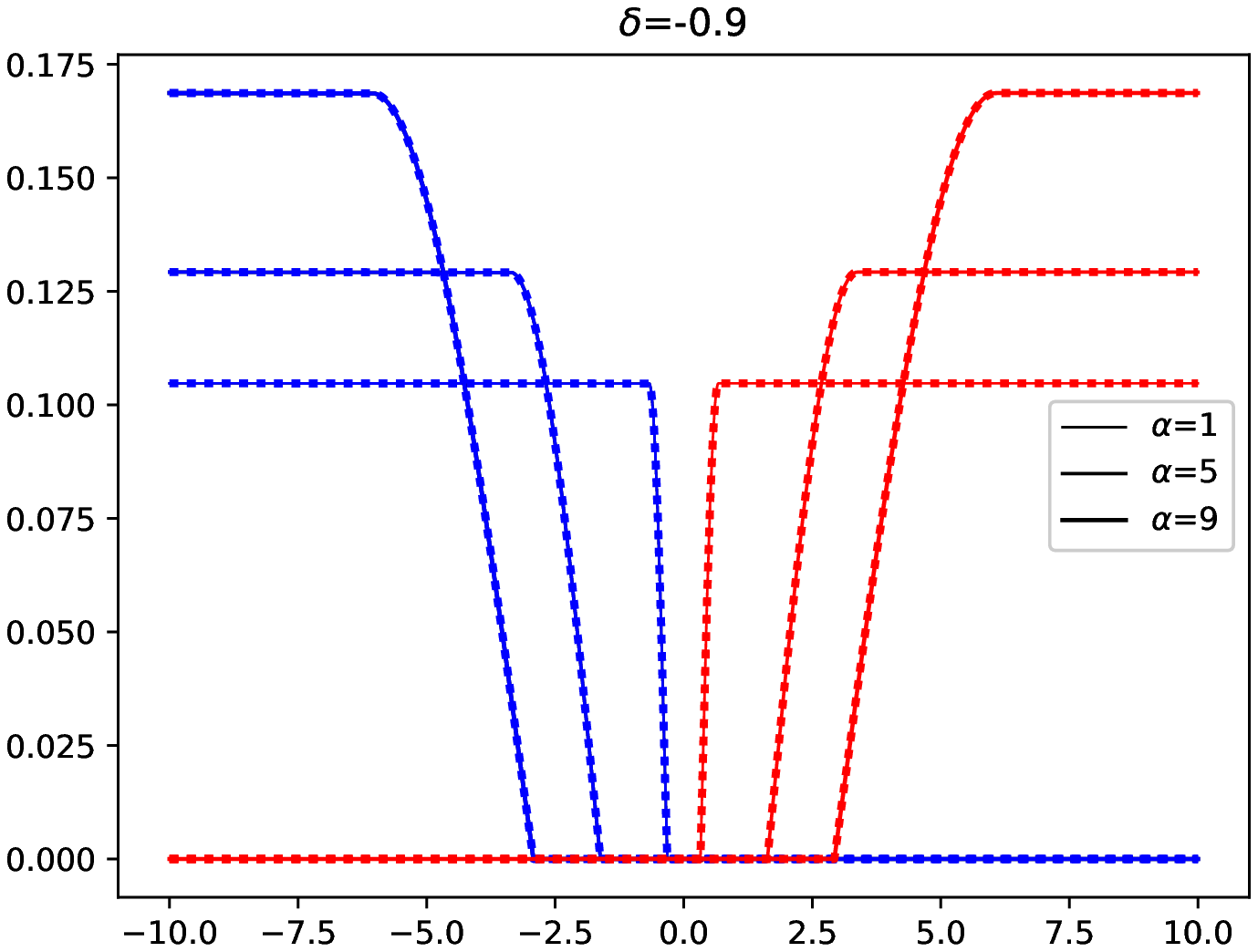}
	}
	\subfigure[Keeping $\alpha=5$ fixed while varying $\delta$.]
	{
		\includegraphics[width=0.45\textwidth]{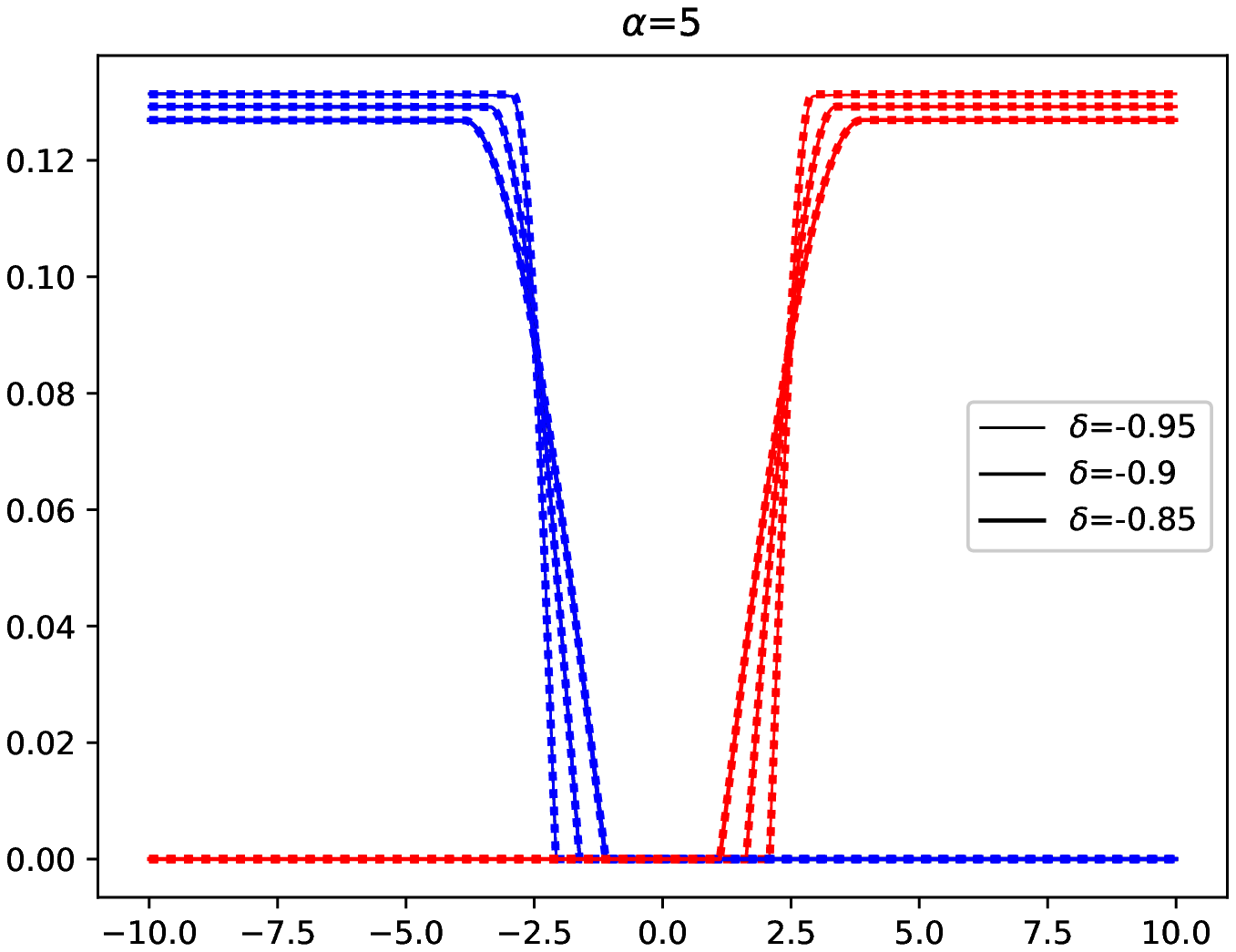}
	}
	\caption{Indicator function as kernel, $K={\bar K}$. The straight lines correspond to the solutions of the numerical minimisation initialised by $\rho(x) = \frac12\indicator_{[-2,0]}(x)$ and $\eta(x)=\frac12 \indicator_{[0,2]}(x)$, the dotted lines correspond to the critical points computed in Section \ref{sec:gapstudy}, cf. Proposition \ref{thm:gaps}.}
	\label{fig:gap_indicator0}
\end{figure}
%
%
\begin{figure}[ht!]
	\centering
	\subfigure[$r=r(\delta)$.]
	{
		\includegraphics[width=0.45\textwidth]{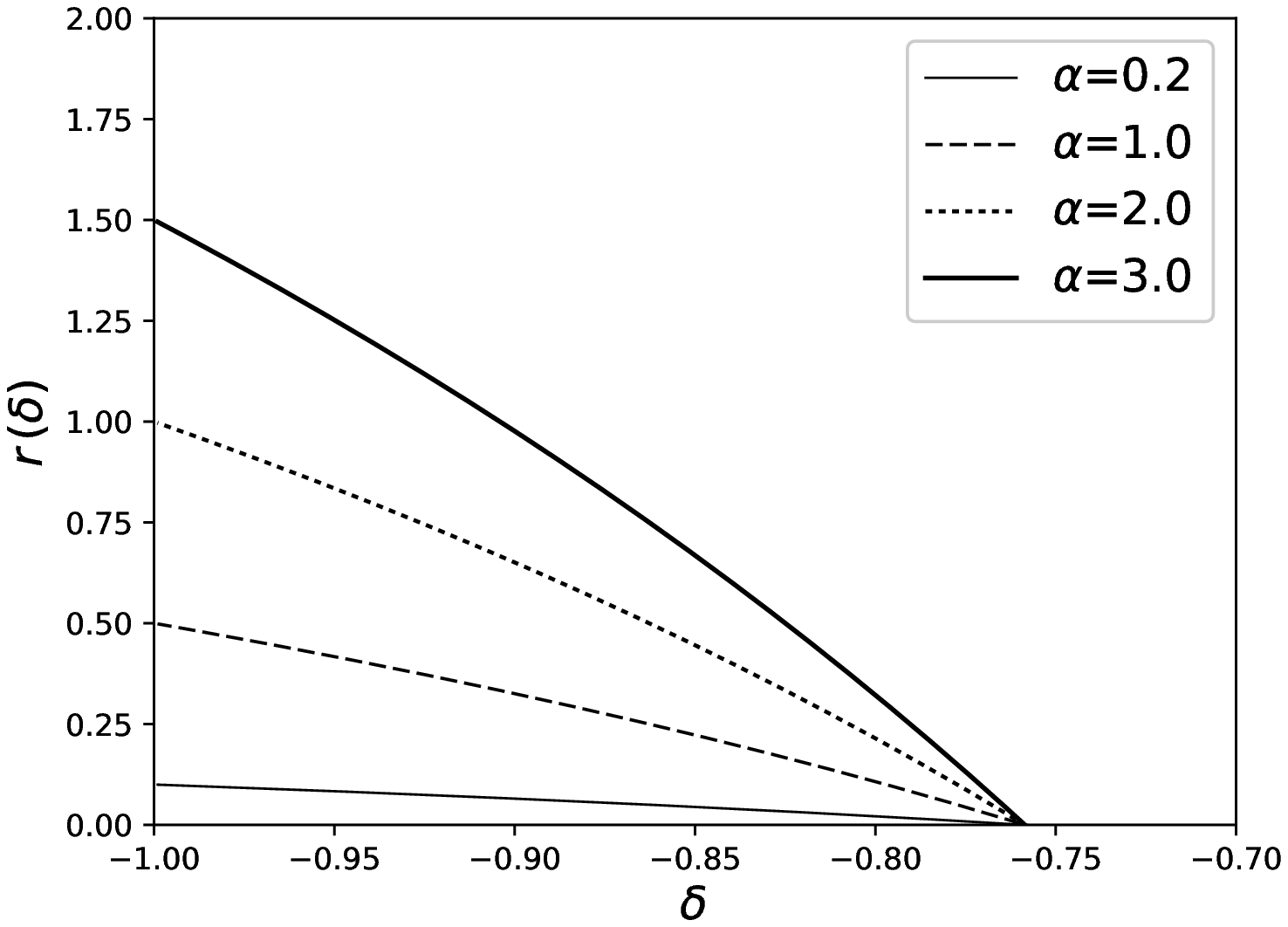}
	}
	\subfigure[$r=r(\alpha)$.]
	{
		\includegraphics[width=0.45\textwidth]{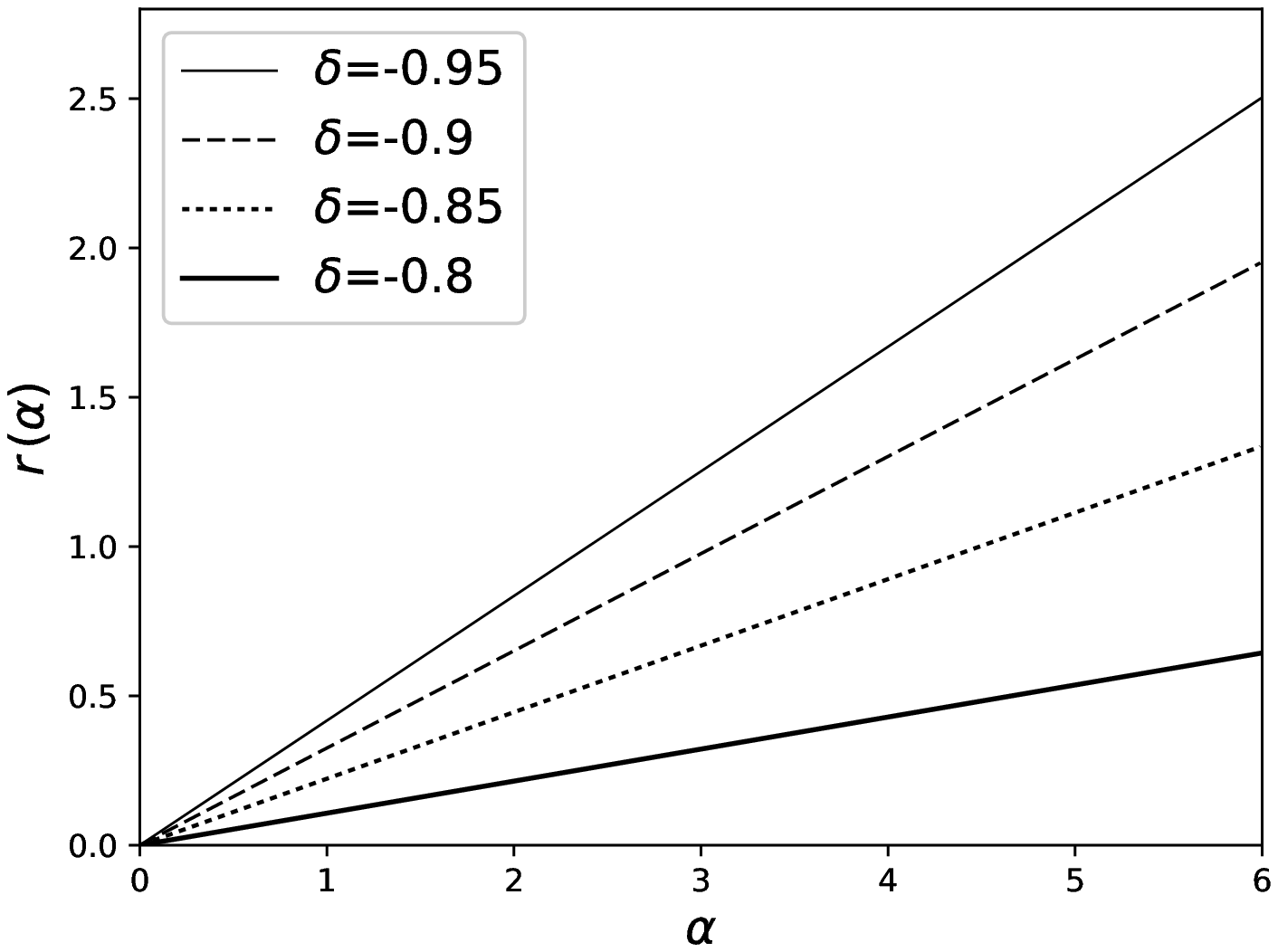}
	}
	\caption{Case of the indicator kernel: the gap width, $r$, as a function of $\alpha$ and $\delta$. Note that for fixed $\alpha$, the critical value is $\delta=-0.7585$, cf. Remark \ref{rem:criticaldelta}, Eq. \eqref{eq:criticaldelta}.}
	\label{fig:gap_indicator}
\end{figure}

\begin{figure}[ht!]
	\centering
	\subfigure[Keeping $\delta=0-0.9$ fixed while varying $\alpha$.]
	{
		\includegraphics[width=0.45\textwidth]{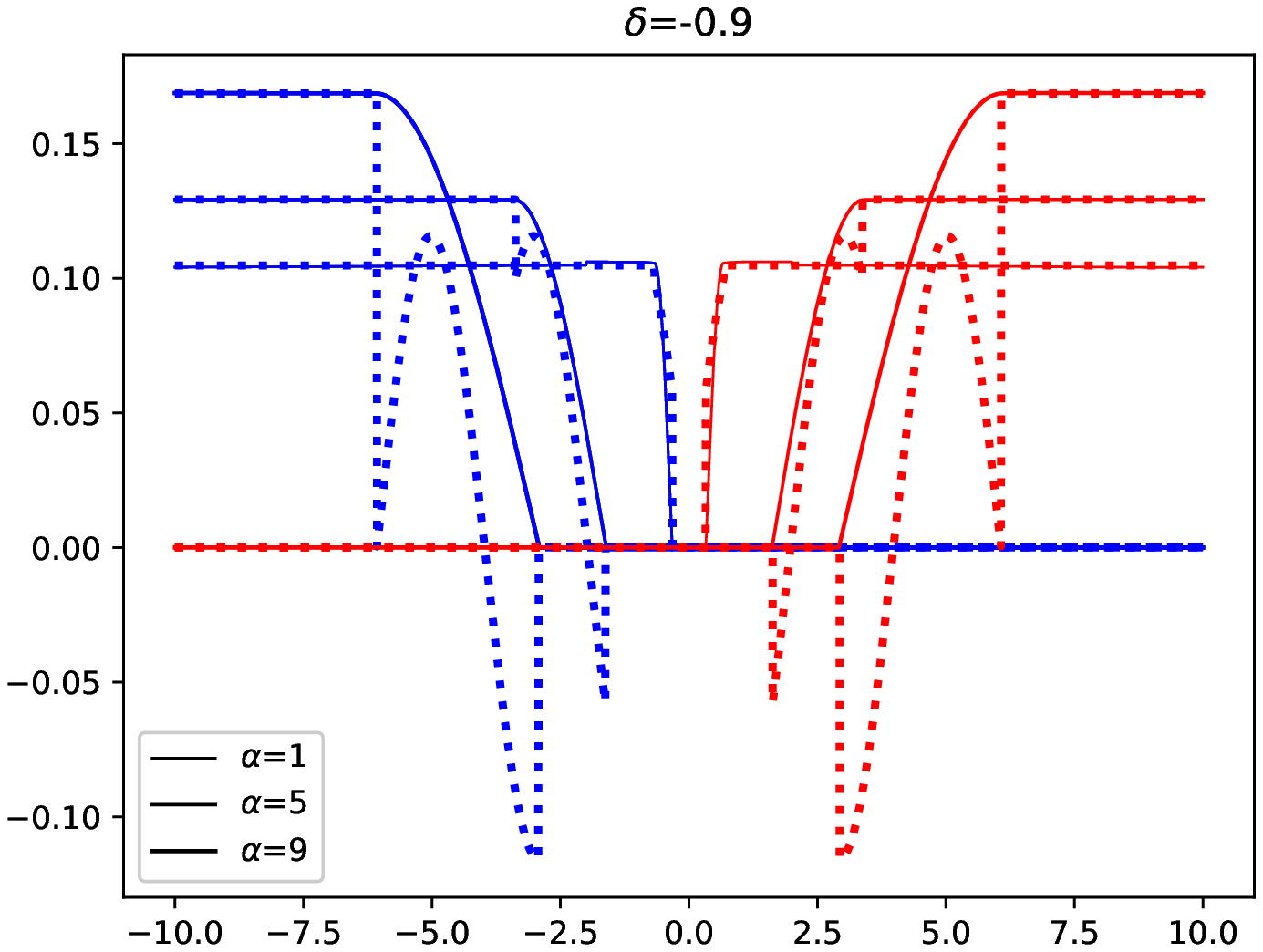}
	}
	\subfigure[Keeping $\alpha=5$ fixed while varying $\delta$.]
	{
		\includegraphics[width=0.45\textwidth]{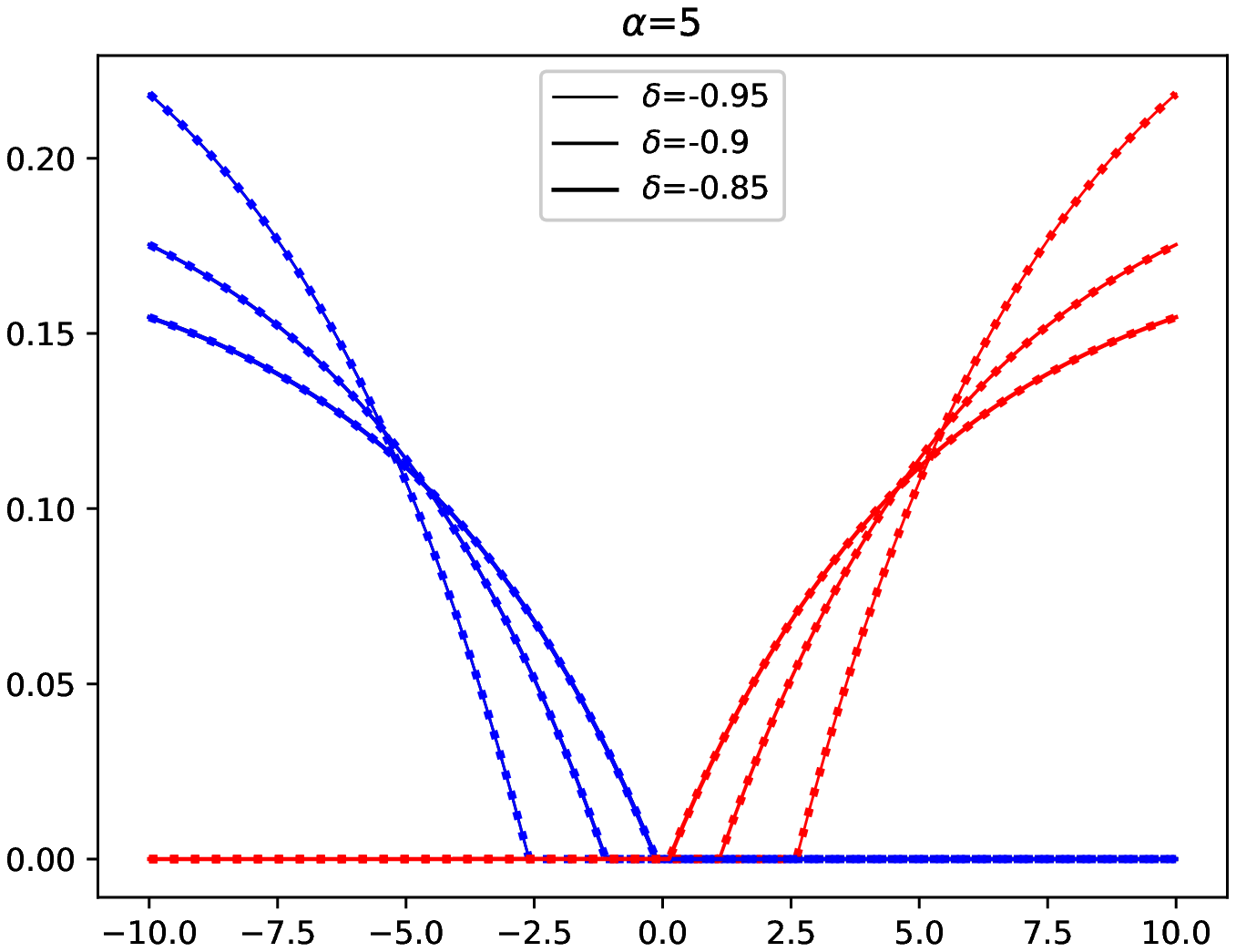}
	}
	\caption{Picard function as kernel, $K={\tilde K}$. The straight lines correspond to the solutions of the numerical minimisation initialised by  $\rho(x) = \frac12\indicator_{[-2,0]}(x)$ and $\eta(x)=\frac12 \indicator_{[0,2]}(x)$, the dotted lines correspond to the critical points computed in Section \ref{sec:gapstudy}, Proposition \ref{thm:gaps}.}
	\label{fig:gap_picard0}
\end{figure}

\begin{figure}[ht!]
	\centering
	\subfigure[$r=r(\delta)$.]
	{
		\includegraphics[width=0.45\textwidth]{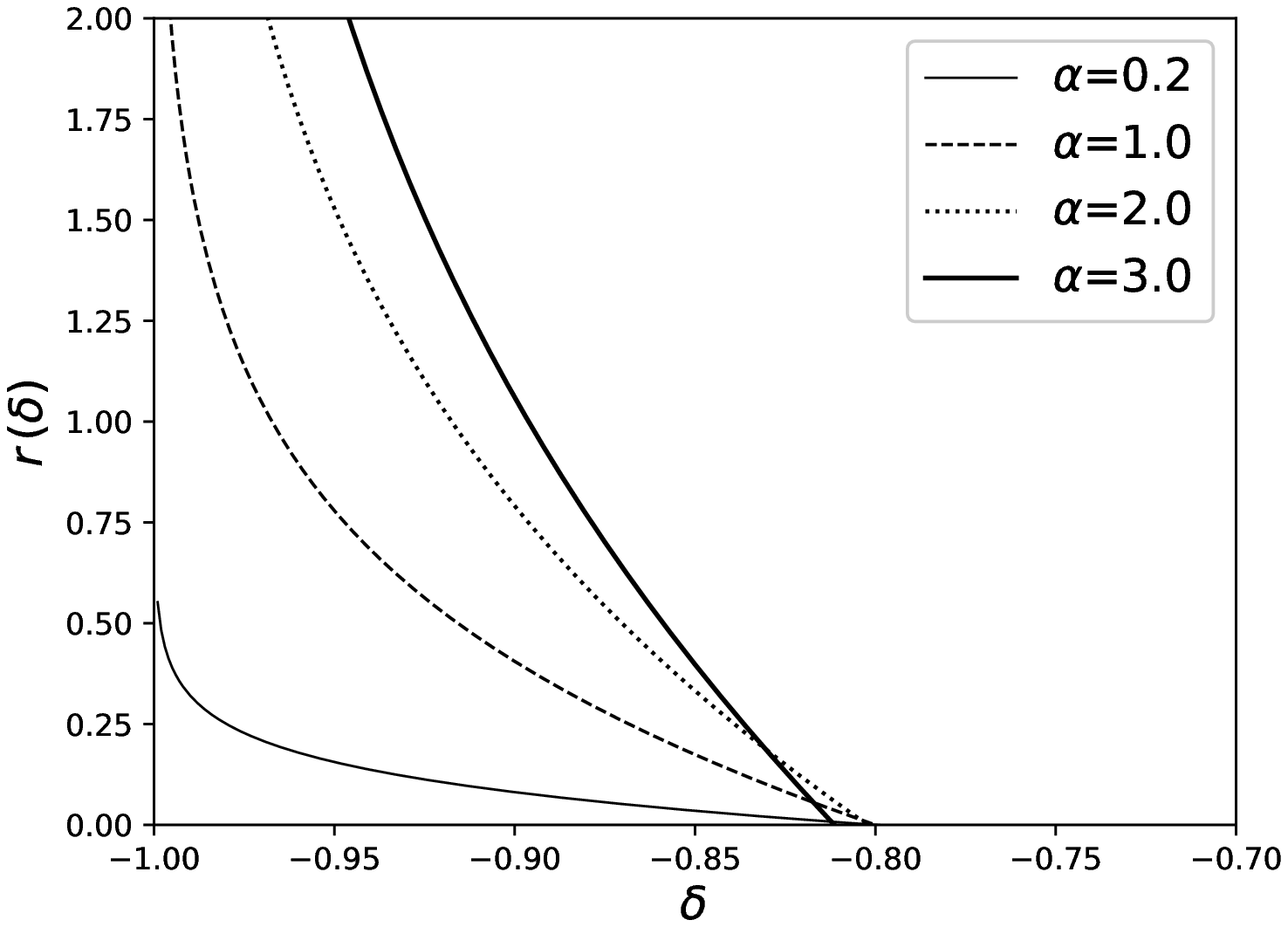}
	}
	\subfigure[$r=r(\alpha)$.]
	{
		\includegraphics[width=0.45\textwidth]{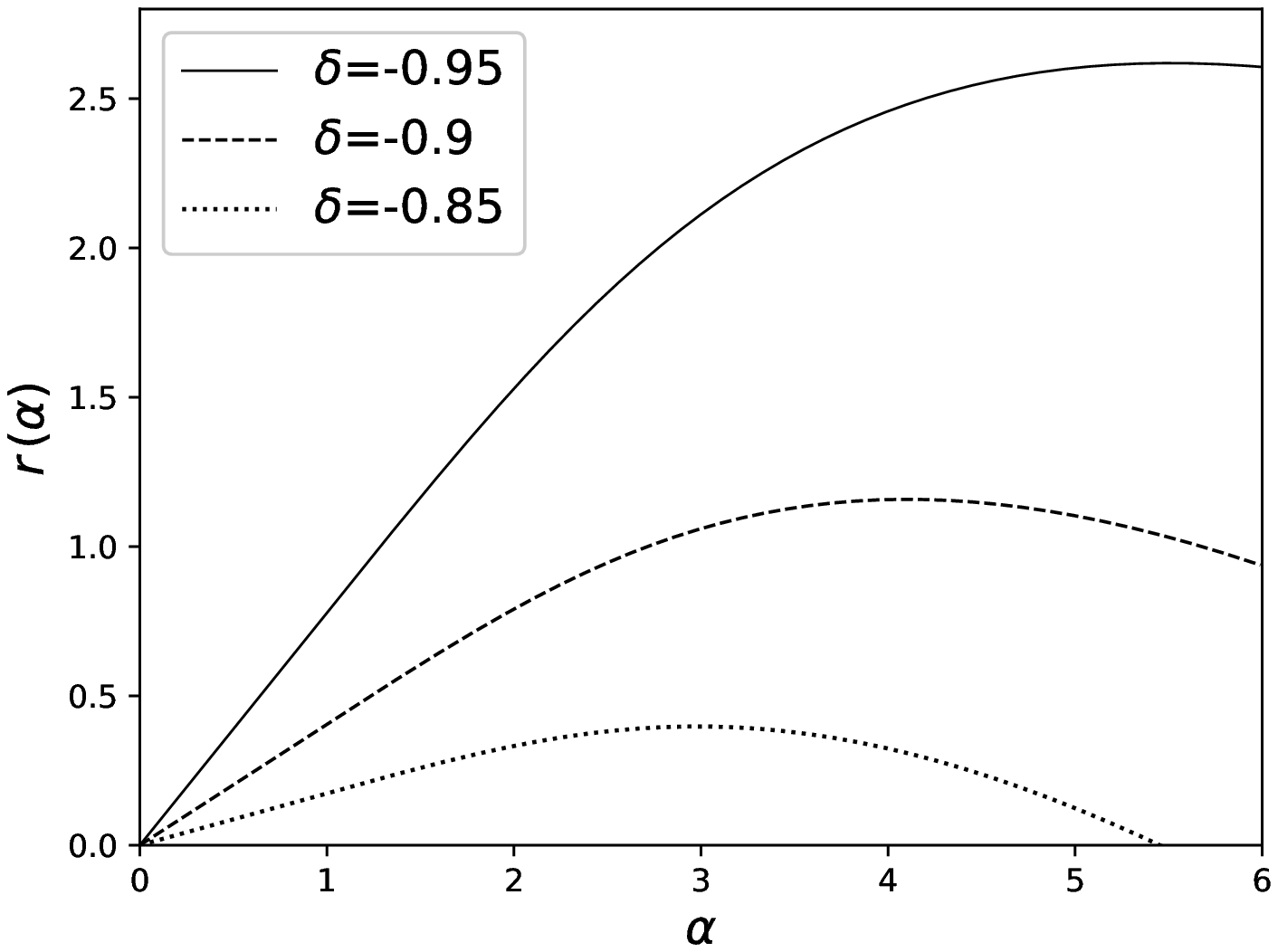}
	}
	\caption{Case of the Picard kernel: the gap width, $r$, depends on both $\alpha$ and $\delta$, cf. Proposition \ref{thm:gaps}.}
	\label{fig:gap_picard}
\end{figure}

\section*{Acknowledgements}
The work of MB has been supported by ERC via Grant EU FP 7 - ERC Consolidator Grant 615216 LifeInverse. MB and JFP acknowledge support by the German Science Foundation DFG via EXC 1003 Cells in Motion Cluster of Excellence, M\"unster. JAC was partially supported by the Royal Society by a Wolfson Research Merit Award and the EPSRC grant EP/P031587/1.  MS acknowledges the funding by the `Santander Mobility Award' and `Cells in Motion'. 

\bibliographystyle{abbrv}
\def\cprime{$'$}

\end{document}